\documentclass[12pt]{amsart}
\usepackage{amsmath,amsfonts,amssymb,amsthm,amstext,pgf,graphicx,hyperref,verbatim,lmodern,ytableau,textcomp,color,young,youngtab,tikz,tikz-cd,bbold}
\usepackage[foot]{amsaddr}
\usepackage[mathscr]{euscript}
\definecolor{asparagus}{rgb}{0.0, 0.5, 0.0}
\theoremstyle{plain}
\newtheorem{thm}{Theorem}[section]
\newtheorem{lem}[thm]{Lemma}
\newtheorem{prop}[thm]{Proposition}
\newtheorem{cor}[thm]{Corollary}
\newtheorem{rem}[thm]{Remark}

\theoremstyle{definition}
\newtheorem{defn}{Definition}[section]

\newcommand{\Gr}{G_n}

\newcommand{\ynirr}{\mathcal{Y}_n(\widehat{G}_n)}
\newcommand{\irrG}{\widehat{G}_n}

\DeclareMathOperator{\1}{id}
\DeclareMathOperator{\trivial}{\mathbb{1}}
\DeclareMathOperator{\Tr}{trace}
\DeclareMathOperator{\dimension}{dim}
\DeclareMathOperator{\Gn}{\mathcal{G}_n}
\DeclareMathOperator{\G1}{e}
\DeclareMathOperator{\y}{\mathcal{Y}}
\DeclareMathOperator{\yn}{\mathcal{Y}_n}
\DeclareMathOperator{\yx}{\mathcal{Y}(X)}
\DeclareMathOperator{\ynx}{\mathcal{Y}_n(X)}

\DeclareMathOperator{\tab}{tab}

\DeclareMathOperator{\Indf}{\mathfrak{Ind}}
\DeclareMathOperator{\M}{\mathcal{M}}

\DeclareMathOperator{\El}{\text{N}}
\title{}
\begin{document}
\title{Cutoff phenomenon for the warp-transpose top with random shuffle}
\author{Subhajit Ghosh}
\address[Subhajit Ghosh]{Department of Mathematics, Bar-Ilan University, Ramat-Gan 5290002}
\email{gsubhajit@alum.iisc.ac.in}
\keywords{random walk, complete monomial group, mixing time, cutoff, Young-Jucys-Murphy elements}
\makeatletter
\@namedef{subjclassname@2020}{%
	\textup{2020} Mathematics Subject Classification}
\makeatother
\subjclass[2020]{60J10, 60B15, 60C05.}
\begin{abstract}
Let $\{G_n\}_1^{\infty}$ be a sequence of non-trivial finite groups. In this paper, we study the properties of a random walk on the complete monomial group $\Gr\wr S_n$ generated by the elements of the form $(\G1,\dots,\G1,g;\1)$ and $(\G1,\dots,\G1,g^{-1},\G1,\dots,\G1,g;(i,n))$ for $g\in \Gr,\;1\leq i< n$. We call this the warp-transpose top with random shuffle on $\Gr\wr S_n$. We find the spectrum of the transition probability matrix for this shuffle. We prove that the mixing time for this shuffle is $O\left(n\log n+\frac{1}{2}n\log (|\Gr|-1)\right)$. We show that this shuffle exhibits $\ell^2$-cutoff at $n\log n+\frac{1}{2}n\log (|\Gr|-1)$ and total variation cutoff at $n\log n$.
\end{abstract}
\maketitle
\section{Introduction}\label{intro}
The number of shuffles required to mix up a deck of cards is the main topic of interest for card shuffling problems. It has received considerable attention in the last few decades. The card shuffling problems can be described as random walks on the symmetric group. The generalisation by replacing the symmetric group with other finite groups is also a well-studied topic in probability theory \cite{S}. A random walk converges to a unique stationary distribution subject to certain natural conditions. For a convergent random walk, the mixing time (number of steps required to reach the stationary distribution up to a given tolerance) is one of the main topics of interest. It is helpful to know the eigenvalues and eigenvectors of the transition matrix to study the convergence rate of random walks. In general, convergence rate questions for finite Markov chains are useful in many subjects, including statistical physics, computer science, biology and more \cite{Intro_Application1_DR}.

In the eighties, Diaconis and Shahshahani introduced the use of non-commutative Fourier analysis techniques in their work on the \emph{random transposition shuffle} \cite{DS}. They proved that this shuffle on $n$ distinct cards has total variation cutoff (sharp mixing time; the formal definition of cutoff will be given later) at $\frac{1}{2}n\log n$. The upper bound estimate in this case mainly uses the fact that the total variation distance is at most half of the $\ell^2$-distance, which relies on the spectrum of the transition matrix. After this landmark work, the theory of random walks on finite groups obtained its own independence, its own problems and techniques. Afterwards, some other techniques have come to deal with random walks on finite groups (viz. the coupling argument \cite{A}, the strong stationary time approach \cite{AD1,AD2}). However, the spectral approach became a standard technique for answering mixing time questions for random walk on finite groups \cite{D1}. For a random walk model with known cutoff, the natural question appears on how the transition occurs at cutoff. In 2020, Teyssier \cite{T-limit} studied the \emph{limit profile} for the random transposition model; providing a precise description of the transition at cutoff. Later Nestoridi and Olesker-Taylor \cite{N-limit} generalised Teyssier's result to reversible Markov chains. In recent times, Nestoridi further develops the previous result by studying the limit profile for the \emph{transpose top with random shuffle} (the shuffling algorithm is given in the next paragraph) \cite{N-star_limit}. This present paper focuses on obtaining the sharp mixing time; the limit profile computation will be considered in future work.

Our model is mainly inspired by the transpose top with random shuffle on the symmetric group $S_n$ \cite{FOW}. Given a deck of $n$ distinct cards, this shuffle chooses a card from the deck uniformly at random and transposes it with the top card. This shuffle exhibits total variation cutoff at $n\log n$ \cite{D1,D2}. The transpose top with random shuffle has been recently generalised by the author to the cards with two orientations known as the \emph{flip-transpose top with random shuffle} on the \emph{hyperoctahedral group} $B_n$ \cite{FTTR}. The flip-transpose top with random shuffle on $B_n$ has total variation cutoff at $n\log n$. In the extended abstract \cite{FPSAC20}, the author has introduced a generalisation of the flip-transpose top with random shuffle to the complete monomial group $G_n\wr S_n$, and announced result on total variation cutoff under the restriction \emph{$|G_n|=o(n^{\delta})$ for all $\delta>0$}. In this work, we further generalise it by removing the restriction on the size of $G_n$, that motivates us to consider both $\ell^2$-distance and total variation distance. Moreover, if $\log(|\Gr|-1)\neq o(\log n)$, then our model provides an example where the spectral analysis fails for obtaining sharp total variation mixing time; in other words, the $\ell^2$-bound of the total variation distance is not sufficient enough for computing sharp total variation mixing time. Thus we consider both $\ell^2$-distance and total variation distance in the present paper. For another notable random walk on the complete monomial groups we mention the work of Schoolfield Jr. \cite{SchJr1}, which is a generalisation of the random transposition model to $G\wr S_n$ for a finite group $G$. However this walk was generated by the probability measure which is constant on the conjugacy classes. On the other hand, the generating measure of our model is not constant on conjugacy classes. In general, it is not easy to study a random walk generated by a probability measure which is not constant on the conjugacy classes (cf. \cite{R2R,FOW,TT2R,FTTR}). For other random walks on the complete monomial group see \cite{SchJr3,Oliver's_thesis}. Before describing our random walk model, let us first recall the definition of the complete monomial group.
\begin{defn}
		Let $G$ be a finite group and $S_n$ be the symmetric group of permutations of elements of the set $[n]:=\{1,2,\dots,n\}$. The \emph{complete monomial group} is the wreath product of $G$ with $S_n$, denoted by $G\wr S_n$, and can be described as follows: The elements of $G\wr S_n$ are $(n+1)$-tuples $(g_1,g_2,\dots,g_n;\pi)$ where $g_i\in G$ and $\pi\in S_n$. The multiplication in $G\wr S_n$ is given by $(g_1,\dots,g_n;\pi)(h_1,\dots,h_n;\eta)=(g_1h_{\pi^{-1}(1)},\dots,g_nh_{\pi^{-1}(n)};\pi\eta)$. Therefore $(g_1,\dots,g_n;\pi)^{-1}=(g_{\pi(1)}^{-1},\dots,g_{\pi(n)}^{-1};\pi^{-1})$.
\end{defn}
Now let $\{G_n\}_1^{\infty}$ be a sequence of non-trivial finite groups. We consider the complete monomial groups $\Gn:=\Gr\wr S_n$ for each positive integer $n$. Let $\G1$ be the identity of $\Gr$ and $\1$ be the identity of $S_n$. For an element $\pi\in S_n$, let $\pi:=(\G1,\dots,\G1;\pi)\in \Gn$ and for $g\in \Gr$, let 
\begin{align*}
	g^{(i)}:=(\G1,\dots,\G1,&\;\underset{\uparrow}{g},\G1,\dots,\G1;\1)\in \Gn.\\[-1ex]
	&i\text{th position.}
\end{align*}
Unless otherwise stated from now on, $(\G1,\dots,\G1,g^{-1},\G1,\dots,\G1,g;(i,n))$ denotes the element of $\Gn$ with $g^{-1}$ in $i$th position and $g$ in $n$th position, for $g\in \Gr,\;1\leq i < n$. One can check that $(g^{-1})^{(i)} g^{(n)}(i,n)$ is equal to $(\G1,\dots,\G1,g^{-1},\G1,\dots,\G1,g;(i,n))$ for $g\in \Gr,\;1\leq i < n$.

In this work we consider a random walk on the complete monomial group $\Gn$ driven by a probability measure $P$, defined as follows:
\begin{equation}\label{def of P}
	P(x)=
	\begin{cases}
		\frac{1}{n|\Gr|},&\text{ if }x=(\G1,\dots,\G1,g;\1)\text{ for }g\in \Gr,\\
		\frac{1}{n|\Gr|},&\text{ if }x=(\G1,\dots,\G1,g^{-1},\G1,\dots,\G1,g;(i,n))\text{ for }g\in \Gr,\;1\leq i < n,\\
		0,&\text{ otherwise}.
	\end{cases}
\end{equation}
We call this the \emph{warp-transpose top with random shuffle} because at most times the $n$th component is multiplied by $g$ and the $i$th component is multiplied by $g^{-1}$ simultaneously, $g\in \Gr$, $1\leq i<n$. We now give a combinatorial description of this model as follows:
\begin{figure}[h]
	\begin{tabular}{ccc}
		$\begin{array}{ccc}
			& & \textcolor{asparagus}{1}\textcolor{blue}{2}\textcolor{asparagus}{3}\textcolor{red}{4}\textcolor{blue}{5}\textcolor{red}{9}\textcolor{red}{7}\textcolor{blue}{8}\textcolor{red}{6}\\
			\textcolor{asparagus}{1}\textcolor{blue}{2}\textcolor{asparagus}{3}\textcolor{red}{4}\textcolor{blue}{5}\hspace*{-0.5ex}
			\begin{minipage}{0.5cm}
				\begin{tikzpicture}
					\hspace*{0.075cm}\draw[fill=white, draw=black] (0,0) ellipse (0.15 and 0.2) node[pos=0.5] {\textcolor{asparagus}{6}};
				\end{tikzpicture}
			\end{minipage}
			\textcolor{red}{7}\textcolor{blue}{8}\textcolor{blue}{9}&
			\hspace*{0.06cm}\begin{array}{c}
				\rotatebox[origin=c]{40}{\textcolor{asparagus}{$\xrightarrow{\hspace*{0.75cm}}$}}\\
				\rotatebox[origin=c]{0}{\textcolor{red}{$\xrightarrow{\hspace*{0.75cm}}$}}\\
				\rotatebox[origin=c]{320}{\textcolor{blue}{$\xrightarrow{\hspace*{0.75cm}}$}}
			\end{array}
			&\textcolor{asparagus}{1}\textcolor{blue}{2}\textcolor{asparagus}{3}\textcolor{red}{4}\textcolor{blue}{5}\textcolor{blue}{9}\textcolor{red}{7}\textcolor{blue}{8}\textcolor{asparagus}{6}\\
			&& \textcolor{asparagus}{1}\textcolor{blue}{2}\textcolor{asparagus}{3}\textcolor{red}{4}\textcolor{blue}{5}\textcolor{asparagus}{9}\textcolor{red}{7}\textcolor{blue}{8}\textcolor{blue}{6}
		\end{array}$
		&\quad\quad\quad
		$\begin{array}{ccc}
			& & \textcolor{asparagus}{1}\textcolor{blue}{2}\textcolor{asparagus}{3}\textcolor{red}{4}\textcolor{blue}{5}\textcolor{asparagus}{6}\textcolor{red}{7}\textcolor{blue}{8}\textcolor{red}{9}\\
			\textcolor{asparagus}{1}\textcolor{blue}{2}\textcolor{asparagus}{3}\textcolor{red}{4}\textcolor{blue}{5}\textcolor{asparagus}{6}\textcolor{red}{7}\textcolor{blue}{8}\hspace*{-0.5ex}
			\begin{minipage}{0.5cm}
				\begin{tikzpicture}
					\hspace*{0.06cm}\draw[fill=white, draw=black] (0,0) ellipse (0.15 and 0.2) node[pos=0.5] {\textcolor{blue}{9}};
				\end{tikzpicture}
			\end{minipage}&
			\begin{array}{c}
				\rotatebox[origin=c]{40}{\textcolor{asparagus}{$\xrightarrow{\hspace*{0.75cm}}$}}\\
				\rotatebox[origin=c]{0}{\textcolor{red}{$\xrightarrow{\hspace*{0.75cm}}$}}\\
				\rotatebox[origin=c]{320}{\textcolor{blue}{$\xrightarrow{\hspace*{0.75cm}}$}}
			\end{array}
			&\textcolor{asparagus}{1}\textcolor{blue}{2}\textcolor{asparagus}{3}\textcolor{red}{4}\textcolor{blue}{5}\textcolor{asparagus}{6}\textcolor{red}{7}\textcolor{blue}{8}\textcolor{blue}{9}\\
			&& \textcolor{asparagus}{1}\textcolor{blue}{2}\textcolor{asparagus}{3}\textcolor{red}{4}\textcolor{blue}{5}\textcolor{asparagus}{6}\textcolor{red}{7}\textcolor{blue}{8}\textcolor{asparagus}{9}
		\end{array}$\\
		$(a)$&\quad\quad\quad$(b)$
	\end{tabular}
	\caption{Transitions for the warp-transpose top with random shuffle on $\mathbb{Z}_3\wr S_{9}$. $\mathbb{Z}_3$ is the additive group of integers modulo $3$, consists of the colours \textcolor{red}{red}, \textcolor{asparagus}{green} and \textcolor{blue}{blue} such that \textcolor{red}{red} represents the identity element. $(a)$ shows transitions when the sixth card is chosen and $(b)$ shows transitions when the last card is chosen. }\label{Figure:a_typical_transition9}
\end{figure}

Let $\mathcal{A}_n(G)$ denote the set of all arrangements of $n$ coloured cards in a row such that the colours of the cards are indexed by the set $G$. For example, if $\mathbb{Z}_2$ denotes the additive group of integers modulo $2$, then elements of $\mathcal{A}_n(\mathbb{Z}_2)$ can be identified with the elements of $B_n$ (the hyperoctahedral group). For $g,h\in G$, by saying \emph{update the colour $g$ using colour $h$} we mean the colour $g$ is updated to colour $g\cdot h$. Elements of $\Gn$ can be identified with the elements of $\mathcal{A}_n(\Gr)$ as follows: The element $(g_1,\dots,g_n;\pi)\in\Gn$ is identified with the arrangement in $\mathcal{A}_n(\Gr)$ such that the label of the $i$th card is $\pi(i)$, and its colour is $g_{\pi(i)}$, for each $i\in[n]$. Given an arrangement of coloured cards in $\mathcal{A}_n(\Gr)$, the warp-transpose top with random shuffle on $\Gn$ is the following: Choose a positive integer $i$ uniformly from $[n]$. Also choose a colour $g$ uniformly from $\Gr$, independent of the choice of the integer $i$.
\begin{enumerate}
	\item If $i=n$: update the colour of the $n$th card using colour $g$.
	\item If $i<n$: first transpose the $i$th and $n$th cards. Then  simultaneously update the colour of the $i$th card using colour $g$ and update the colour of the $n$th card using colour $g^{-1}$.
\end{enumerate}
The flip-transpose top with random shuffle on the hyperoctahedral group serves the case when $|\Gr|=2$ for all $n$ \cite{FTTR}. We now state the main theorems of this paper.
\begin{thm}\label{thm:mixing_main}
	The $($total variation and $\ell^2\text{-})$ mixing time for the warp-transpose top with random shuffle on $\Gn$ is $O\left(n\log n+\frac{1}{2}n\log (|\Gr|-1)\right)$.
\end{thm}
\begin{thm}\label{thm:ell-2}
	The warp-transpose top with random shuffle on $\Gn$ presents $\ell^2$-cutoff at $n\log n+\frac{1}{2}n\log (|\Gr|-1)$.
\end{thm}
\begin{thm}\label{thm:cutoff_main}
	The warp-transpose top with random shuffle on $\Gn$ exhibits total variation cutoff at time $n\log n$.
\end{thm}
In view of Theorem \ref{thm:cutoff_main}, the distribution after $n\log n(1+o(1))$ transitions is close to uniform in total variation distance. On the other hand, if $\log(|\Gr|-1)\neq o(\log n)$, then Theorem \ref{thm:ell-2} and
\[\left(n\log n+\frac{1}{2}n\log (|\Gr|-1)\right)(1-o(1))>n\log n(1+o(1))\]
 says that the distribution after $n\log n(1+o(1))$ transitions is far from uniform in $\ell^2$-distance. Therefore, the standard spectral approach (mainly uses the fact that the total variation distance is at most half of $\ell^2$-distance) for obtaining the total variation mixing time fails when $\log(|\Gr|-1)\neq o(\log n)$. The proof of Theorem \ref{thm:mixing_main} and Theorem \ref{thm:ell-2} will be presented in Section \ref{sec:mixingtime order and ell^2 cutoff}, and the proof of Theorem \ref{thm:cutoff_main} will be presented at the end of Section \ref{sec:lower bound}.
 
Let us recall some concepts and terminologies from representation theory of finite group and discrete time Markov chains with finite state space to make this paper self contained. Readers from representation theoretic background may skip Subsection \ref{subsection:Rep_th_bckground} and from Probabilistic background may skip Subsection \ref{subsection:Markov_chain_background}.
\subsection{Representation theory of finite group}\label{subsection:Rep_th_bckground}
Let $G$ be a finite group and $V$ be a finite dimensional complex vector space. Also let $GL(V)$ be the set of all invertible linear operators on $V$. A \emph{linear representation} $\rho$ of $G$ is a homomorphism from $G$ to $GL(V)$. Sometimes this representation is also denoted by the pair $(\rho,V)$. The dimension of the vector space $V$ is called the \emph{dimension} $d_{\rho}$ of the representation. $V$ is called the \emph{$G$-module} corresponding to the representation $\rho$ in this case. Let $\mathbb{C}[G]$ be the group algebra consisting of complex linear combinations of elements of $G$. In particular taking $V=\mathbb{C}[G]$, we define the \emph{right regular representation} $R:G\longrightarrow GL(\mathbb{C}[G])$ of $G$ by 
\[R(g)\left(\sum_{h\in G}C_hh\right)=\sum_{h\in G}C_hhg^{-1},\text{ where }C_h\in\mathbb{C}.\]
A vector subspace $W$ of $V$ is said to be \emph{stable} ( or `\emph{invariant}') under $\rho$ if $\rho(g)\left(W\right)\subset W$ for all $g$ in $G$. The representation $\rho$ is \emph{irreducible} if $V$ is non-trivial and $V$ has no non-trivial proper stable subspace. Two representations $(\rho_1,V_1)$ and $(\rho_2,V_2)$ of $G$ are said to be \emph{isomorphic} if there exists an invertible linear map $T:V_1\rightarrow V_2$ such that the following diagram commutes for all $g\in G$:
\[\begin{tikzcd}
V_1\arrow{r}{\rho_1(g)}\arrow[swap]{d}{T} & V_1\arrow{d}{T}\\
V_2\arrow{r}{\rho_2(g)} & V_2
\end{tikzcd}\]

For each $g\in G,\;\rho(g)$ can also be thought of as an invertible complex matrix of size $d_{\rho}\times d_{\rho}$. The trace of the matrix $\rho(g)$ is said to be the \emph{character} value of $\rho$ at $g$ and is denoted by $\chi^{\rho}(g)$. It can be easily seen that the character values are constants on conjugacy classes, hence characters are class functions. If $\overline{\chi^{\rho}(g)}$ denote the complex conjugate of $\chi^{\rho}(g)$, then one can check that $\chi^{\rho}(g^{-1})=\overline{\chi^{\rho}(g)}$ for all $g\in G$. Let $\mathscr{C}(G)$ be the complex vector space of class functions of $G$. Then a `standard' inner product $\langle\cdot,\cdot\rangle$ on $\mathscr{C}(G)$ is defined as follows:
\[\langle\phi,\psi\rangle=\frac{1}{|G|}\sum_{g\in G}\phi(g)\psi(g^{-1})\quad\text{ for }\quad\phi,\psi\in\mathscr{C}(G).\]
An important theorem in this context is the following \cite[Theorem 6]{Serre}: \emph{The characters corresponding to the non-isomorphic irreducible representations of $G$ form an $\langle\cdot,\cdot\rangle$-orthonormal basis of $\mathscr{C}(G)$}.

If $V_1\otimes V_2$ denotes the tensor product of the vector spaces $V_1$ and $V_2$, then the tensor product of two representations $\rho_1: G\rightarrow GL(V_1)$ and $\rho_2: G\rightarrow GL(V_2)$ is a representation denoted by $(\rho_1\otimes\rho_2,V_1\otimes V_2)$ and defined by,
\[(\rho_1\otimes\rho_2)(g)(v_1\otimes v_2)=\rho_1(g)(v_1)\otimes\rho_2(g)(v_2) \text{ for }v_1\in V_1,v_2\in V_2\text{ and }g\in G.\]
We use some results from representation theory of finite groups without recalling the proof. For details about finite group representation see \cite{Amri,Sagan,Serre}. 
\subsection{Discrete time Markov chain with finite state space}\label{subsection:Markov_chain_background}
Let $\Omega$ be a finite set. A sequence of random variables $X_0,X_1,\dots$ is a \emph{discrete time Markov chain with state space $\Omega$ and transition matrix $M$} if for all $x,y\in\Omega$, all $k>1$, and all events $H_{k-1}:=\underset{0\leq s<k}{\cap}\{X_s=x_s\}$ satisfying $\mathbb{P}(H_{k-1}\cap\{X_k=x\})>0$, we have 
\begin{equation}\label{eq:Markov_property}
\mathbb{P}(X_{k+1}=y\mid H_{k-1}\cap\{X_k=x\})=M(x,y).
\end{equation}
Equation \eqref{eq:Markov_property} says that given the present, the future is independent of the past. Let $\mathscr{D}_k$ denote the distribution after $k$ transitions, i.e. $\mathscr{D}_k$ is the row (probability) vector $\left(\mathbb{P}(X_k=x)\right)_{x\in\Omega}$. Then $\mathscr{D}_k=\mathscr{D}_{k-1}M$ for all $k\geq1$,  which implies $\mathscr{D}_k=\mathscr{D}_0M^k$. In particular if the chain starts at $x\in\Omega$, then its distribution after $k$ transitions is $\mathscr{D}_k=\delta_xM^k$, i.e. $\mathbb{P}(X_k=y\mid X_0=x)=M^k(x,y)$. Here $\delta_x$ is defined on $\Omega$ as follows:
\[\delta_x(u)=
\begin{cases}
1 & \text{ if } u=x,\\
0 & \text{ if } u\neq x.
\end{cases}\]
A Markov chain is said to be \emph{irreducible} if it is possible for the chain to reach any state starting from any state using only transitions of positive probability. The \emph{period} of a state $x\in\Omega$ is defined to be the greatest common divisor of the set of all times when it is possible for the chain to return to the starting state $x$. \emph{The period of all the states of an irreducible Markov chain are the same} \cite[Lemma 1.6]{LPW}. An irreducible Markov chain is said to be \emph{aperiodic} if the common period for all its states is $1$. A probability distribution $\Pi$ is said to be a \emph{stationary distribution} of the Markov chain if $\Pi M=\Pi$. Any irreducible Markov chain possesses a unique stationary distribution $\Pi$ with $\Pi(x)>0$ for all $x\in\Omega$ \cite[Proposition 1.14]{LPW}. Moreover if the chain is aperiodic then $\mathscr{D}_k\longrightarrow\Pi$ as $k\longrightarrow\infty$ \cite[Theorem 4.9]{LPW}. For an irreducible chain, we first define the \emph{$\ell^2$-distance} between the distribution after $k$ transitions and the stationary distribution.
\begin{defn}\label{def:ell-2_distance}
	Let $\mathscr{D}_k$ denote the distribution after $k$ transitions of an irreducible discrete time Markov chain with finite state space $\Omega$, and $\Pi$ denote its stationary distribution. Then the \emph{$\ell^2$-distance} between $\mathscr{D}_k$ and $\Pi$ is defined by \[\left|\left|\mathscr{D}_k-\Pi\right|\right|_{2}:=\left(\sum_{x\in\Omega}\left|\frac{\mathscr{D}_k(x)}{\Pi(x)}-1\right|^2\Pi(x)\right)^{\frac{1}{2}}.\]
\end{defn}
We now define the \emph{total variation distance} between two probability measures.
\begin{defn}\label{def:Total_varition_distance}
	Let $\mu$ and $\nu$ be two probability measures on $\Omega$. The \emph{total variation distance} between $\mu$ and $\nu$ is defined by \[\left|\left|\mu-\nu\right|\right|_{\text{TV}}:=\sup_{A\subset\Omega}|\mu(A)-\nu(A)|.\]
It can be easily seen that $\left|\left|\mu-\nu\right|\right|_{\text{TV}}=\frac{1}{2}\sum_{x\in\Omega}|\mu(x)-\nu(x)|$ (see \cite[Proposition 4.2]{LPW}).
\end{defn}
For an irreducible and aperiodic chain the interesting topic is the minimum number of transitions $k$ required to reach near the stationarity $\Pi$ up to a certain level of tolerance $\varepsilon>0$. We first define the maximal $\ell^2$-distance (respectively total variation distance) between the distribution after $k$ transitions and the stationary distribution as follows:
\[d_{2}(k):=\underset{x\in\Omega}{\max}\;\left|\left|M^k(x,\cdot)-\Pi\right|\right|_{2}\;\;\left(\text{respectively }d_{\text{TV}}(k):=\underset{x\in\Omega}{\max}\;\left|\left|M^k(x,\cdot)-\Pi\right|\right|_{\text{TV}}\right).\]
For $\varepsilon>0$, the $\ell^2$-mixing time (respectively total variation mixing time) with tolerance level $\varepsilon$ is defined by
\[\tau_{\text{mix}}(\varepsilon):=\min\;\{k:d_{2}(k)\leq\varepsilon\}\;\;(\text{respectively }t_{\text{mix}}(\varepsilon):=\min\;\{k:d_{\text{TV}}(k)\leq\varepsilon\}).\]
Most of the notations of this subsection are borrowed from \cite{LPW}.
\subsection{Non-commutative Fourier analysis and random walks on finite groups}\label{subsection:rwfg}
Let $\;p\text{ and }q$ be two probability measures on a finite group $G$. We define the \emph{convolution $p*q$} of $p$ and $q$ by $(p*q)(x):=\sum_{y\in G}p(xy^{-1})q(y)$. The \emph{Fourier transform} $\widehat{p}$ of $p$ at the right regular representation $R$ is defined by the matrix $\sum_{x\in G}p(x)R(x)$. The matrix $\widehat{p}(R)$ can be thought of as the action of the group algebra element $\sum_{g\in G}p(g)g^{-1}$ on $\mathbb{C}[G]$ by multiplication on the right. It can be easily seen that $\widehat{(p*q)}(R)=\widehat{p}(R)\widehat{q}(R)$.  
	
A \emph{random walk on a finite group $G$ driven by a probability measure $p$} is a Markov chain with state space $G$ and transition probabilities $M_p(x,y)=\;p(x^{-1}y)$, $x,y\in G$. It can be easily seen that the transition matrix $M_p$ is $\widehat{p}(R)$ and the distribution after $k$th transition will be $p^{*k}$ (convolution of $p$ with itself $k$ times) i.e., the probability of getting into state $y$ starting at state $x$ after $k$ transitions is $p^{*k}(x^{-1}y)$. One can easily check that \emph{the random walk on $G$ driven by $p$ is irreducible if and only if the support of $p$ generates $G$} \cite[Proposition 2.3]{S}. \emph{The stationary distribution for an irreducible random walk on $G$ driven by $p$, is the uniform distribution $U_G$ on $G$} (since $\sum_{x\in G}M_p(x,y)=\sum_{x\in G}p(x^{-1}y)=\sum_{z\in G}p(z)=1,\;z=x^{-1}y$ for all $y\in G$). From now on, the uniform distribution on group $G$ will be denoted by $U_G$. 
For the random walk on $G$ driven by $p$, it is enough to focus on $\left|\left|p^{*k}-U_G\right|\right|_{2}$ and $\left|\left|p^{*k}-U_G\right|\right|_{\text{TV}}$ because,
\begin{align*}
&\left|\left|M_p^k(x,\cdot)-U_G\right|\right|_{\text{TV}}=\left|\left|M_p^k(y,\cdot)-U_G\right|\right|_{\text{TV}}\\
&\left|\left|M_p^k(x,\cdot)-U_G\right|\right|_{2}=\left|\left|M_p^k(y,\cdot)-U_G\right|\right|_{2}
\end{align*}
for any two elements $x,y\in G$. We now define the cutoff  phenomenon for a sequence of random walks on finite groups.
\begin{defn}\label{cutoff defn}
	Let $\{\mathscr{G}_n\}_{0}^{\infty}$ be a sequence of finite groups. For each $n$, let $p_n$ be a probability measure on $\mathscr{G}_n$ such that support of $p_n$ generate $\mathscr{G}_n$. Consider the sequence of irreducible and aperiodic random walk on $\mathscr{G}_n$ driven by $p_n$. We say that the \emph{$\ell^2$-cutoff phenomenon} (respectively \emph{total variation cutoff phenomenon}) holds for the family $\{(\mathscr{G}_n,p_n)\}_0^{\infty}$ if there exists a sequence $\{\mathfrak{T}_n\}_0^{\infty}$ of positive real numbers tending to infinity as $n\rightarrow\infty$, such that the following hold:
	\begin{enumerate}
		\item For any $\epsilon\in (0,1)$ and $k_n=\lfloor(1+\epsilon)\mathfrak{T}_n\rfloor$, \[\lim_{n\rightarrow\infty}\left|\left|p_n^{*k_{n}}-U_{\mathscr{G}_n}\right|\right|_{2}=0\;\left(\text{respectively }\lim_{n\rightarrow\infty}\left|\left|p_n^{*k_{n}}-U_{\mathscr{G}_n}\right|\right|_{\text{TV}}=0\right),\]
		\item For any $\epsilon\in (0,1)$ and $k_n=\lfloor(1-\epsilon)\mathfrak{T}_n\rfloor$, \[\lim_{n\rightarrow\infty}\left|\left|p_n^{*k_{n}}-U_{\mathscr{G}_n}\right|\right|_{2}=\infty\;\left(\text{respectively }\lim_{n\rightarrow\infty}\left|\left|p_n^{*k_{n}}-U_{\mathscr{G}_n}\right|\right|_{\text{TV}}=1\right).\]
	\end{enumerate}
	Here $\lfloor x\rfloor$ denotes the floor  of $x$ (the largest integer less than or equal to $x$).
\end{defn}
Informally, we will say that $\{(\mathscr{G}_n,p_n)\}_0^{\infty}$ has an $\ell^2$-cutoff (respectively total variation cutoff) at time $\mathfrak{T}_n$. This says that for sufficiently large $n$ the leading order term in the mixing time does not depend on the tolerance level $\varepsilon(>0)$. In other words the distribution after $k$ transitions is very close to the stationary distribution if $k=\mathfrak{T}_n$ but too far from the stationary distribution if $k<\mathfrak{T}_n$. Although most of the cases the cutoff phenomenon depend on the multiplicity of the second largest eigenvalue of the transition matrix \cite{DCutoff}, sometimes the behaviour is different. For the random-to-random shuffle \cite{R2R}, many eigenvalues are almost equal to the second largest eigenvalue, and they impact the (total variation) cutoff time. We now see that the random walk of our concern is irreducible and aperiodic.
\begin{prop}\label{irreducibility and aperiodicity}
	The warp-transpose top with random shuffle on $\Gn$ is irreducible and aperiodic.
\end{prop}
\begin{proof}
	The support of $P$ is $\Gamma=\{(g^{-1})^{(i)}g^{(n)}(i,n),g^{(n)}\mid g\in \Gr,1\leq i<n\}$ and it can be easily seen that $\{g^{(k)},(i,n)\mid g\in \Gr,\;1\leq k\leq n,\;1\leq i<n\}$ is a generating set of $\Gn$. 
	\begin{equation}\label{eq:Irredc+aperiodi}
	\begin{split}
	& (g^{-1})^{(n)}\left((g^{-1})^{(i)}g^{(n)}(i,n)\right)g^{(n)}=(i,n)\text{ for each }1\leq i<n\text{ and }g\in \Gr,\\
	& (k,n)g^{(n)}(k,n)=g^{(k)} \text{ for each }1\leq k\leq n\text{ and for all }g\in \Gr.
	\end{split}
	\end{equation}
	Thus \eqref{eq:Irredc+aperiodi} implies $\Gamma$ generates $\Gn$ and hence the warp-transpose top with random shuffle on $\Gn$ is irreducible. Moreover given any $\pi\in \Gn$, the set of all times when it is possible for the chain to return to the starting state $\pi$ contains the integer $1$ (as support of $P$ contains the identity element of $\Gn$). Therefore the period of the state $\pi$ is $1$ and hence from irreducibility all the states of this chain have period $1$. Thus this chain is aperiodic.
\end{proof}
Proposition \ref{irreducibility and aperiodicity} says that the warp-transpose top with random shuffle on $\Gn$ converges to the uniform distribution $U_{\Gn}$ as the number of transitions goes to infinity.
In Section \ref{sec:representation} we will find the spectrum of $\widehat{P}(R)$. We will prove Theorems \ref{thm:mixing_main} and \ref{thm:ell-2} in Section \ref{sec:mixingtime order and ell^2 cutoff}. In Section \ref{sec:upper bound}, we will obtain an upper bound of $\left|\left|P^{*k}-U_{\Gn}\right|\right|_{\text{TV}}$ using the coupling argument. In Section \ref{sec:lower bound}, the lower bound of $\left|\left|P^{*k}-U_{\Gn}\right|\right|_{\text{TV}}$ will be discussed and Theorem \ref{thm:cutoff_main} will be proved.
\section{Spectrum of the transition matrix}\label{sec:representation}
In this section we find the eigenvalues of the transition matrix $\widehat{P}(R)$, the Fourier transform of $P$ at the right regular representation $R$ of $\Gn$. To find the eigenvalues of $\widehat{P}(R)$ we will use the representation theory of the wreath product $\Gn$ of $\Gr$ with the symmetric group $S_n$. First we briefly discuss the representation theory of $G\wr S_n$, following the notation from \cite{MS}. We refer to the exposition \cite{MS} for more details on the representation theory of $G\wr S_n$.

A \emph{partition} $\lambda$ of a positive integer $n$ (denoted $\lambda\vdash n$) is a weakly decreasing finite sequence $(\lambda_1,\cdots,\lambda_r)$ of positive integers such that $\sum_{i=1}^{r}\lambda_i=n$. The partition $\lambda$ can be pictorially visualized as a left-justified arrangement of $r$ rows of boxes with $\lambda_i$ boxes in the $i$th row, $1\leq i\leq r$. This pictorial arrangement of boxes is known as the \emph{Young diagram} of $\lambda$. For example there are five partitions of the positive integer $4$ viz. $(4)$, $(3,1)$, $(2,2)$, $(2,1,1)$ and $(1,1,1,1)$. The Young diagrams corresponding to the partitions of $4$ are shown in Figure \ref{fig: Yng_diag_with_4_boxes}.
\begin{defn}\label{def: yx and ynx}
Let $\y$ denote the set of all Young diagrams (there is a unique Young diagram with zero boxes) and $\yn$ denote the set of all Young diagrams with $n$ boxes. For example, elements of $\mathcal{Y}_4$ are shown in Figure \ref{fig: Yng_diag_with_4_boxes}. For a finite set $X$, we define $\yx=\{\mu:\mu\text{ is a map from }X\text{ to }\y\}$. For $\mu\in\yx$, define $||\mu||=\sum_{x\in X}|\mu(x)|$, where $|\mu(x)|$ is the number of boxes of the Young diagram $\mu(x)$ and define $\ynx=\{\mu\in\yx:||\mu||=n\}$.
\end{defn}
\begin{figure}[h]
	\centering
	$\begin{array}{cclll}
	\yng(4)&\hspace{0.5cm}\yng(3,1)& \hspace{0.5cm}\yng(2,2) & \hspace{0.5cm}\yng(2,1,1) & \hspace{0.75cm}\yng(1,1,1,1)\\
	(4)\;&\;\quad(3,1)&\quad\;(2,2)&\quad(2,1,1)&\;(1,1,1,1)
	\end{array}$
	\caption{Young diagrams with $4$ boxes.}\label{fig: Yng_diag_with_4_boxes}
\end{figure}
Let $n$ be a fixed positive integer. Let $\widehat{G}$ denote the (finite) set of all non-isomorphic irreducible representations of $G$. Given $\sigma\in\widehat{G}$, we denote by $W^{\sigma}$ the corresponding irreducible $G$-module (the space for the corresponding irreducible representation of $G$). Elements of $\y(\widehat{G})$ are called \emph{Young $G$-diagrams} and elements of $\yn(\widehat{G})$ are called \emph{Young $G$-diagrams with $n$ boxes}. For example, if $n=10$ and $G=\mathbb{Z}_{10}$ (the additive group of integers modulo $10$), then an element of $\mathcal{Y}_{10}(\widehat{\mathbb{Z}}_{10})$ is given in Figure \ref{fig:Young_G_diagram}. Let $\mu\in\y$. A \emph{Young tableau} of shape $\mu$ is obtained by taking the Young diagram $\mu$ and filling its $|\mu|$ boxes (bijectively) with the numbers $1,2,\dots,|\mu|$. A Young tableau is said to be \emph{standard} if the numbers in the boxes strictly increase along each row and each column of the Young diagram of $\mu$. The set of all standard Young tableaux of shape $\mu$ is denoted by $\tab(\mu)$. Elements of $\tab((3,1))$ are listed in Figure \ref{fig: Stab_of_shape_(3,1)}. Let $\mu\in\y(\widehat{G})$. A \emph{Young $G$-tableau} of shape $\mu$ is obtained by taking the Young $G$-diagram $\mu$ and filling its $||\mu||$ boxes (bijectively) with the numbers $1,2,\dots,||\mu||.$ A Young $G$-tableau is said to be \emph{standard} if the numbers in the boxes strictly increase along each row and each column of all Young diagrams occurring in $\mu$. Let $\tab_{G}(n,\mu)$, where $\mu\in\yn(\widehat{G})$, denote the set of all standard Young $G$-tableaux of shape $\mu$ and let $\tab_G(n)=\cup_{\mu\in\yn(\widehat{G})}\tab_{G}(n,\mu)$. An element of $\tab_{\mathbb{Z}_{10}}\left(10,\mu\right)$ is given in Figure \ref{fig:Young_G_tableau}, the shape $\mu\;\left(\in\y_{10}(\widehat{\mathbb{Z}}_{10})\right)$ is given in Figure \ref{fig:Young_G_diagram}.
\begin{defn}\label{def:b_(i)+r_T(i)}
	Let $T\in\tab_G(n)$ and $i\in[n]$. \emph{If $i$ appears in the Young diagram $\mu(\sigma)$, where $\mu$ is the shape of $T$ and $\sigma\in\widehat{G}$, we write $r_{T}(i)=\sigma$}. For the example given in Figure \ref{fig:Young_G_tableau}, we have $r_{T}(1)=\sigma_2$, $r_{T}(2)=\sigma_2$, $r_{T}(3)=\sigma_8$, $r_{T}(4)=\sigma_1$, $r_{T}(5)=\sigma_{10}$, $r_{T}(6)=\sigma_{1}$, $r_{T}(7)=\sigma_{1}$, $r_{T}(8)=\sigma_8$, $ r_{T}(9)=\sigma_1$, $r_T(10)=\sigma_1$. The \emph{content} of a box in row $p$ and column $q$ of a Young diagram is the integer $q-p$. \emph{Let $b_T(i)$ be the box in $\mu(\sigma)$, with the number $i$ resides and $c(b_T(i))$ denote the content of the box $b_T(i)$}. For the example given in Figure \ref{fig:Young_G_tableau}, we also have $c(b_T(1))=0$, $c(b_T(2))=-1$, $c(b_T(3))=0$, $c(b_{T}(4))=0$, $c(b_{T}(5))=0$, $c(b_T(6))=1$, $c(b_T(7))=-1$, $c(b_{T}(8))=-1$, $c(b_{T}(9))=2$, $c(b_T(10))=0$.
\end{defn}
\begin{figure}[h]
	\centering
	$\mu:=\left(\mu(\sigma_1),\mu(\sigma_2),\dots,\mu(\sigma_{10})\right)=\left(\begin{array}{c}\yng(3,2)\end{array}\hspace*{-0.75ex},\begin{array}{c}\yng(1,1)\end{array},\emptyset,\;\emptyset,\;\emptyset,\;\emptyset,\;\emptyset,\begin{array}{c}\yng(1,1)\end{array},\;\emptyset,\begin{array}{c}\yng(1)\end{array}\right)$
	\caption{An Young $\mathbb{Z}_{10}$-diagram with $10$ boxes $\mu$. Here $\widehat{\mathbb{Z}}_{10}:=\{\sigma_i:1\leq i\leq 10\}$.}\label{fig:Young_G_diagram}
\end{figure}
\begin{figure}[h]
	\centering
	$\begin{array}{ccc}
		\young({{\substack{1}}}{{\substack{2}}}{{\substack{3}}},{{\substack{4}}}) &\quad\young({{\substack{1}}}{{\substack{2}}}{{\substack{4}}},{{\substack{3}}}) & \quad \young({{\substack{1}}}{{\substack{3}}}{{\substack{4}}},{{\substack{2}}})
	\end{array}$
	\caption{Standard Young tableaux of shape $(3,1)$.}\label{fig: Stab_of_shape_(3,1)}
\end{figure}
\begin{figure}[h]
	\centering
	$\mu\rightsquigarrow\left(
	\begin{array}{c}\young({{\substack{4}}}{{\substack{6}}}{{\substack{9}}},{{\substack{7}}}{{\substack{10}}})\end{array}\hspace*{-0.75ex}, \begin{array}{c}\young({{\substack{1}}},{{\substack{2}}})\end{array}\hspace*{-0.75ex},\;\emptyset,\;\emptyset,\;\emptyset,\;\emptyset,\;\emptyset, \begin{array}{c}\young({{\substack{3}}},{{\substack{8}}})\end{array}\hspace*{-0.75ex},\;\emptyset,\begin{array}{c}\young({{\substack{5}}})\end{array}
	\right)$
	\caption{A standard Young $\mathbb{Z}_{10}$-tableaux of shape $\mu$, defined in Figure \ref{fig:Young_G_diagram}.}\label{fig:Young_G_tableau}
\end{figure}
The irreducible representation of $G\wr S_n$ can be parametrised by elements of $\mathcal{Y}_n(\widehat{G})$ \cite[Lemma 6.2 and Theorem 6.4]{MS}. Given $\mu\in\y(\widehat{G})$ and $\sigma\in\widehat{G},\;\mu\downarrow_{\sigma}$ denotes the set of all Young $G$-diagrams obtained from $\mu$ by removing one of the inner corners in the Young diagram $\mu(\sigma)$; see Figure \ref{fig:branching-diagram} for an example. The branching rule \cite[Theorem 6.6]{MS} of the pair $G\wr S_{n-1}\subseteq G\wr S_n$ is given as follows: Let $V^{\mu}$ (respectively $V^{\lambda}$) denote the irreducible $G\wr S_n$-module (respectively $G\wr S_{n-1}$-module) indexed by $\mu\in\yn(\widehat{G})$ (respectively $\lambda\in\mathcal{Y}_{n-1}(\widehat{G})$). Then,
	\begin{equation}\label{eq:GwrS-branching}
		V^{\mu}\big\downarrow_{G\wr S_{n-1}}^{G\wr S_{n}}\cong\underset{\sigma\in\widehat{G}:\;\mu(\sigma)\neq\emptyset}{\oplus}\dimension(W^{\sigma})\left(\underset{\lambda\in\mu\downarrow_{\sigma}}{\oplus}V^{\lambda}\right),
	\end{equation}
where $W^{\sigma}$ is the irreducible $G$-module indexed by $\sigma$. For an illustration of \eqref{eq:GwrS-branching}, consider $\mu\in\y_{10}(\widehat{\mathbb{Z}}_{10})$ from Figure \ref{fig:Young_G_diagram}, and recall $\mu\downarrow_{\sigma_i}$ $(i=1,2,8,10)$ from Figure \ref{fig:branching-diagram}. Then
\[V^{\mu}\big\downarrow^{\mathbb{Z}_{10}\wr S_{10}}_{\mathbb{Z}_{10}\wr S_9}\cong\dimension(W^{\sigma_1})\left(\underset{\lambda\in\mu\downarrow_{\sigma_1}}{\oplus}V^{\lambda}\right)\oplus\dimension(W^{\sigma_2})V^{\mu\downarrow_{\sigma_2}}\oplus\dimension(W^{\sigma_8})V^{\mu\downarrow_{\sigma_8}}\oplus\dimension(W^{\sigma_{10}})V^{\mu\downarrow_{\sigma_{10}}}.\]
Here we have used the same notation for a singleton set and its element.
\begin{figure}[h]
	\centering
	$\begin{array}{c}
		\begin{array}{c}
			\mu\downarrow_{\sigma_1}=\bigg\{\left(\begin{array}{c}\yng(2,2)\end{array}\hspace*{-0.75ex},\begin{array}{c}\yng(1,1)\end{array},\emptyset,\;\emptyset,\;\emptyset,\;\emptyset,\;\emptyset,\begin{array}{c}\yng(1,1)\end{array},\;\emptyset,\begin{array}{c}\yng(1)\end{array}\right),\\
			\hspace*{2.25in}
			\left(\begin{array}{c}\yng(3,1)\end{array}\hspace*{-0.75ex},\begin{array}{c}\yng(1,1)\end{array},\emptyset,\;\emptyset,\;\emptyset,\;\emptyset,\;\emptyset,\begin{array}{c}\yng(1,1)\end{array},\;\emptyset,\begin{array}{c}\yng(1)\end{array}\right)\bigg\}
		\end{array}\;\quad\;\\
		\mu\downarrow_{\sigma_2}=\bigg\{\left(\begin{array}{c}\yng(3,2)\end{array}\hspace*{-0.75ex},\begin{array}{c}\yng(1)\end{array},\emptyset,\;\emptyset,\;\emptyset,\;\emptyset,\;\emptyset,\begin{array}{c}\yng(1,1)\end{array},\;\emptyset,\begin{array}{c}\yng(1)\end{array}\right)\bigg\}\\
		\mu\downarrow_{\sigma_8}=\bigg\{\left(\begin{array}{c}\yng(3,2)\end{array}\hspace*{-0.75ex},\begin{array}{c}\yng(1,1)\end{array},\emptyset,\;\emptyset,\;\emptyset,\;\emptyset,\;\emptyset,\begin{array}{c}\yng(1)\end{array},\;\emptyset,\begin{array}{c}\yng(1)\end{array}\right)\bigg\}\\
		\mu\downarrow_{\sigma_{10}}=\bigg\{\left(\begin{array}{c}\yng(3,2)\end{array}\hspace*{-0.75ex},\begin{array}{c}\yng(1,1)\end{array},\emptyset,\;\emptyset,\;\emptyset,\;\emptyset,\;\emptyset,\begin{array}{c}\yng(1,1)\end{array},\;\emptyset,\;\emptyset\right)\bigg\}\;\quad\;
	\end{array}$
	\caption{$\mu\downarrow_{\sigma}$ for non empty $\mu({\sigma})$, where $\mu\left(\in\mathcal{Y}_{10}(\widehat{\mathbb{Z}}_{10})\right)$ is defined in Figure \ref{fig:Young_G_diagram}.}\label{fig:branching-diagram}\vspace*{-1ex}
\end{figure}
\begin{defn}
	Let $\mathcal{H}_{i,n}(G)$ be the subgroup \[\{(g_1,\dots,g_n,\pi)\in G\wr S_n:\pi(j)=j\text{ for }i+1\leq j\leq n\}\] of $G\wr S_n$ for $0\leq i\leq n$. In particular $\mathcal{H}_{0,n}(G)=\mathcal{H}_{1,n}(G)=G^n$ and $\mathcal{H}_{n,n}(G)=G\wr S_n$.
\end{defn}
	The subgroup $\mathcal{H}_{i,n}(G)$ is isomorphic to $G\wr S_i\times G^{n-i}$ (direct product of $G\wr S_i$ and $G^{n-i}$) by the isomorphism $\Psi:\mathcal{H}_{i,n}(G)\rightarrow G\wr S_i\times G^{n-i}$; sending $(g_1,\dots,g_i,g_{i+1},\dots,g_n;\pi)\in\mathcal{H}_{i,n}(G)$ to $\left((g_1,\dots,g_i;\pi),(g_{i+1},\dots,g_n)\right)\in G\wr S_i\times G^{n-i}$. Here we have used the same notation $\pi$ for permutations of $S_i$ and $S_n$, as $\pi(j)=j$ for $i+1\leq j\leq n$. The irreducible $G\wr S_i\times G^{n-i}$-modules are given by the tensor products of the irreducible $G\wr S_i$-modules and the irreducible $G^{n-i}$-modules \cite[Theorem 10]{Serre}. Therefore we may parametrise the irreducible representations of $\mathcal{H}_{i,n}(G)$ by elements of $\y_i(\widehat{G})\times\widehat{G}^{n-i}$. The branching rule of the pair $\mathcal{H}_{i-1,n}(G)\subseteq \mathcal{H}_{i,n}(G)$ is given as follows: Let $V^{(\mu,\sigma_{i+1},\sigma_{i+2},\dots,\sigma_{n})}$ (respectively $V^{(\lambda,\sigma,\sigma_{i+1},\sigma_{i+2},\dots,\sigma_{n})}$) denote the irreducible $\mathcal{H}_{i,n}(G)$-module (respectively $\mathcal{H}_{i-1,n}(G)$-module) indexed by $(\mu,\sigma_{i+1},\sigma_{i+2},\dots,\sigma_{n})\in\y_i(\widehat{G})\times\widehat{G}^{n-i}$ (respectively $(\lambda,\sigma,\sigma_{i+1},\sigma_{i+2},\dots,\sigma_{n})\in\y_{i-1}(\widehat{G})\times\widehat{G}^{n-i+1}$). Then,
	\begin{equation}\label{eq:H_i,n-branching}
		V^{(\mu,\sigma_{i+1},\sigma_{i+2},\dots,\sigma_{n})}\big\downarrow_{\mathcal{H}_{i-1,n}(G)}^{\mathcal{H}_{i,n}(G)}\cong\underset{\sigma\in\widehat{G}:\;\mu(\sigma)\neq\emptyset}{\oplus}\left(\underset{\lambda\in\mu\downarrow_{\sigma}}{\oplus}V^{(\lambda,\sigma,\sigma_{i+1},\sigma_{i+2},\dots,\sigma_{n})}\right).
	\end{equation}
	In particular, $V^{(\mu)}=V^{\mu}$ (irreducible $G\wr S_n$-module), for $i=n$. To illustrate \eqref{eq:H_i,n-branching}, consider $\mu\in\y_{10}(\widehat{\mathbb{Z}}_{10})$ from Figure \ref{fig:Young_G_diagram}, and recall $\mu\downarrow_{\sigma_i}$ $(i=1,2,8,10)$ from Figure \ref{fig:branching-diagram}. Then
	\[V^{(\mu)}\big\downarrow_{\mathcal{H}_{9,10}(\mathbb{Z}_{10})}^{\mathcal{H}_{10,10}(\mathbb{Z}_{10})}\cong\left(\underset{\lambda\in\mu\downarrow_{\sigma_1}}{\oplus}V^{(\lambda,\sigma_1)}\right)\oplus V^{(\mu\downarrow_{\sigma_2},\sigma_2)}\oplus V^{(\mu\downarrow_{\sigma_8},\sigma_8)}\oplus V^{(\mu\downarrow_{\sigma_{10}},\sigma_{10})}.\]
	Here we have used the same notation for a singleton set and its element.
	\begin{rem}
		Although the simple (multiplicity free) branching of the pair $\mathcal{H}_{i-1,n}(G)\subseteq \mathcal{H}_{i,n}(G)$ was established in \cite[Section 4]{MS}, no branching rule was explained. To see the branching rule \eqref{eq:H_i,n-branching}, we note down the following: A straightforward generalisation of \cite[Theorem 4.3]{AS} provides a proof of the branching rule \eqref{eq:H_i,n-branching} when $i=n$. In view of isomorphism $\Psi$, it suffices to prove \eqref{eq:H_i,n-branching} for $i=n$.
	\end{rem}
\begin{defn}\label{def:YJM_els}
	The (generalised) \emph{Young-Jucys-Murphy} elements $X_1(G),\dots,X_n(G)$ of $\mathcal{H}_{n,n}(G)$ or $\mathbb{C}[G\wr S_n]$ are given by $X_1(G)= 0$ and 
	\begin{align*}
	X_i(G)&=\sum\limits_{k=1}^{i-1}\sum\limits_{g\in G}(g^{-1})^{(k)}g^{(i)}(k,i)=\displaystyle\sum_{k=1}^{i-1}\sum\limits_{g\in G}(g^{-1})^{(k)}(k,i)g^{(k)},\text{ for all }2\leq i\leq n.
	\end{align*}
\end{defn}
Young-Jucys-Murphy elements generates a maximal commuting subalgebra of $\mathbb{C}[G\wr S_n]$ and act like scalars on the \emph{Gelfand-Tsetlin subspaces} of irreducible $G\wr S_n$-modules. We now define Gelfand-Tsetlin subspaces and the Gelfand-Tsetlin decomposition.

Let $\mu\in\widehat{\mathcal{H}_{n,n}}(G)$ and consider the irreducible $\mathcal{H}_{n,n}(G)$-module $V^{\mu}$ (the space for the representation $\mu$). Since the branching is simple (recall \eqref{eq:H_i,n-branching}), the decomposition into irreducible $\mathcal{H}_{n-1,n}(G)$-modules is given by 
\[V^{\mu}=\underset{\lambda}{\oplus}V^{\lambda},\]
where the sum is over all $\lambda\in \widehat{\mathcal{H}_{n-1,n}}(G)$, with $\lambda\nearrow\mu$ (i.e there is an edge from $\lambda$ to $\mu$ in the branching multi-graph), is canonical. Here we note that $\mu\in\yn(\widehat{G})$ and $\lambda\in\y_{n-1}(\widehat{G})\times\widehat{G}$. Iterating this decomposition of $V^{\mu}$ into irreducible $\mathcal{H}_{1,n}(G)$-submodules, i.e.,
\begin{equation}\label{eq:GZ-}
	V^{\mu}=\underset{T}{\oplus}V_{T},
\end{equation}
where the sum is over all possible chains $T=\mu_1\nearrow\mu_2\nearrow\dots\nearrow\mu_n$ with $\mu_i\in \widehat{\mathcal{H}_{i,n}}(G)$ and $\mu_n=\mu$. We call \eqref{eq:GZ-} the \emph{Gelfand-Tsetlin} decomposition of $V^{\mu}$ and each $V_T$ in \eqref{eq:GZ-} a \emph{Gelfand-Tsetlin} subspace of $V^{\mu}$. We note that if $0\neq v_T\in V_T,\text{ then }\mathbb{C}[\mathcal{H}_{i,n}(G)]v_T=V^{\mu_i}$ from the definition of $V_T$.

\begin{thm}[{\cite[Theorem 6.5]{MS}}]\label{thm:action of yjm els}
	Let $\mu\in\yn(\widehat{G})$. Then we may index the Gelfand-Tsetlin subspaces of $V^{\mu}$ by standard Young $G$-tableaux  of shape $\mu$ and write the Gelfand-Tsetlin decomposition as 
	\[V^{\mu}=\underset{T\in\tab_{G}(n,\mu)}{\oplus}V_T,\]
	where each $V_T$ is closed under the action of $G^n$ and as a $G^n$-module, is isomorphic to the irreducible $G^n$-module 
	\[W^{r_T(1)}\otimes W^{r_T(2)}\otimes\dots\otimes W^{r_T(n)}.\]
	For $i=1,\dots,n;$ the eigenvalues of $X_i(G)$ on $V_T$ are given by $\frac{|G|}{\dimension(W^{r_T(i)})}c(b_T(i))$. Here we recall $r_T(i)$ and $c(b_T(i))$ from Definition \ref{def:b_(i)+r_T(i)}, $W^{r_T(i)}$ is the irreducible $G$-module indexed by $r_T(i)$, and $X_i(G)$ is the $i$th Young-Jucys-Murphy element of $\mathcal{H}_{n,n}$.
\end{thm}
\begin{thm}[{\cite[Theorem 6.7]{MS}}]\label{thm:dimen of irr G_n-modules}
	Let $\mu\in\yn(\widehat{G})$. Write the elements of $\widehat{G}$ as $\{\sigma_1,\dots,\sigma_t\}$ and set $\mu^{(i)}=\mu(\sigma_i),\;m_i=|\mu^{(i)}|,d_i=\dimension(W^{\sigma_i})$ for each $1\leq i\leq t$. Then 
	\[\dimension(V^{\mu})=\binom{n}{m_1,\dots,m_t}f^{\mu^{(1)}}\cdots f^{\mu^{(t)}}d_1^{m_1}\cdots d_t^{m_t}.\]
	Here $f^{\mu^{(i)}}$ denotes the number of standard Young tableau of shape $\mu^{(i)}$, for each $1\leq i\leq t$.
\end{thm}
\begin{lem}\label{lem:only G-action}
	Let G be a finite group and $\rho\in\widehat{G}$. If $W^{\rho}\;($respectively $\chi^{\rho})$ denotes the irreducible $G$-module $($respectively character$)$ and $d_{\rho}$ is the dimension of $W^{\rho}$, then the action of the group algebra element $\sum_{g\in G}g$ on $W^{\rho}$ is given by the following scalar matrix
	\[\sum_{g\in G}g=\frac{|G|}{d_{\rho}}\langle\chi^{\rho},\chi^{\trivial}\rangle I_{d_{\rho}}.\]
	Here $I_{d_{\rho}}$ is the identity matrix of order $d_{\rho}\times d_{\rho}$ and $\trivial$ be the trivial representation of $G$.
\end{lem}
\begin{proof}
	It is clear that $\sum_{g\in G}g$ is in the centre of $\mathbb{C}[G]$. Therefore by Schur's lemma (\cite[Proposition 4]{Serre}), we have $\sum_{g\in G}g=c I_{d_{\rho}}$ for some $c\in\mathbb{C}$. The value of $c$ can be obtained by equating the traces of $\sum_{g\in G}g$ and $c I_{d_{\rho}}$.
\end{proof}
\begin{rem}
	Our focus will be on $\mathcal{H}_{n,n}(\Gr)$ i.e. $\Gr\wr S_n$ for the sequence of subgroups
	\[\mathcal{H}_{1,n}(\Gr)\subseteq\cdots\subseteq\mathcal{H}_{i,n}(\Gr)\subseteq\cdots \subseteq\mathcal{H}_{n,n}(\Gr).\]
	For simplicity we write the Young-Jucys-Murphy elements $X_i(\Gr)$ of $\Gr\wr S_n$ (i.e. $\Gn$) as $X_i$ for $1\leq i\leq n$. Thus Theorems \ref{thm:action of yjm els} and \ref{thm:dimen of irr G_n-modules} are applicable to $\Gn$.
\end{rem}
	Let $t:=|\irrG|$ and $\irrG:=\{\sigma_1,\dots,\sigma_t\}$, where $\sigma_1=\trivial$ (the trivial representation of $\Gr$). We write $\mu\left(\in\ynirr\right)$ as the tuple $(\mu^{(1)},\dots,\mu^{(t)})$, where $\mu^{(i)}:=\mu(\sigma_i)$ for each $1\leq i\leq t$. We also denote $m_i:=|\mu^{(i)}|,\;W^{\sigma_i}:=$ the irreducible $\Gr$-module corresponding to $\sigma_i$ and $d_i=\dim(W^{\sigma_i})$ for each $1\leq i\leq t$. Thus $t,\;\sigma_i,\;\mu^{(i)},\;m_i,\;W^{\sigma_i}$, and $d_i$ depend on $\Gr$ i.e., on $n$. To avoid notational complication the dependence of $t,\;\sigma_i,\;\mu^{(i)},\;m_i,\;W^{\sigma_i}$, and $d_i$ on $n$ is suppressed. We note that for $T\in\tab_{\Gr}(n,\mu)$ the dimension of $V_T$ is $d_1^{m_1}\cdots d_t^{m_t}$.
\begin{thm}\label{thm: eigenvalues on irre}
	For each $\mu=(\mu^{(1)},\dots,\mu^{(t)})\in\ynirr$, let $\widehat{P}(R)\big|_{V^{\mu}}$ denote the restriction of $\widehat{P}(R)$ to the irreducible $\Gn$-module $V^{\mu}$. Then the eigenvalues of $\widehat{P}(R)\big|_{V^{\mu}}$ are given by,
	\[\frac{1}{n\dimension(W^{r_T(n)})}\left(c(b_T(n))+\langle\chi^{r_T(n)},\chi^{\trivial}\rangle\right),\text{ with multiplicity }\dimension(V_T)=d_1^{m_1}\cdots d_t^{m_t}\]
	for each $T\in\tab_{\Gr}(n,\mu)$. 
\end{thm}
\begin{proof}
	We first find the eigenvalues of $X_n+\sum_{g\in \Gr}(\G1,\dots,\G1,g;\1)$. Let $I_{\dimension(V_T)}$ denote the identity matrix of order $\dimension(V_T)\times\dimension(V_T)$. Then from Theorem \ref{thm:action of yjm els} we have
	\begin{equation}\label{eigenvalu_thm_eqn1}
		V^{\mu}=\underset{T\in\tab_{\Gr}(n,\mu)}{\oplus}V_T\quad\text{ and }\quad X_n\big|_{V_T}=\frac{|\Gr|}{\dimension(W^{r_T(n)})}c(b_T(n))I_{\dimension(V_T)}.
	\end{equation}
	Again from Theorem \ref{thm:action of yjm els} and Lemma \ref{lem:only G-action} we have
	\begin{equation}\label{eigenvalu_thm_eqn2}
		\sum_{g\in \Gr}(\G1,\dots,\G1,g;\1)\big|_{V_T}=\frac{|\Gr|}{\dimension(W^{r_T(n)})}\langle\chi^{r_T(n)},\chi^{\trivial}\rangle I_{\dimension(V_T)}.
	\end{equation} 
	\[\text{We recall }\widehat{P}(R)=\frac{1}{n|\Gr|}\sum_{g\in \Gr}\hspace*{-0.75ex}\left(R\left((\G1,\dots,\G1,g;\1)\right)+\sum_{i=1}^{n-1}R\left((\G1,\dots,\G1,g^{-1},\G1,\dots,\G1,g;(i,n))\right)\right).\]
	Therefore $n|\Gr|\widehat{P}(R)$ is the action of $X_n+\sum_{g\in \Gr}(\G1,\dots,\G1,g;\1)$ on $\mathbb{C}[\Gn]$ by multiplication on the right. Since $\dimension(V_T)=d_1^{m_1}\cdots d_t^{m_t}$, the theorem follows from \eqref{eigenvalu_thm_eqn1} and \eqref{eigenvalu_thm_eqn2}.
\end{proof}
\begin{rem}\label{rem:all eigenvalues}
	In the regular representation of a finite group, each irreducible representation occurs with multiplicity equal to its dimension \emph{\cite[section 2.4]{Serre}}. Therefore, Theorems \ref{thm:dimen of irr G_n-modules} and \ref{thm: eigenvalues on irre} provide the eigenvalues of $\widehat{P}(R)$.
\end{rem}
\section{Order of the mixing time and $\ell^2$-cutoff}\label{sec:mixingtime order and ell^2 cutoff}
In this section, using the spectrum of the transition matrix $\widehat{P}(R)$, we find the upper bounds of $\left|\left|P^{*k}-U_{\Gn}\right|\right|_{2}$ and $\left|\left|P^{*k}-U_{\Gn}\right|\right|_{\text{TV}}$ when $k\geq n\log n+\frac{1}{2}n\log (|\Gr|-1)+Cn,\;C>0$. We also prove Theorems \ref{thm:mixing_main} and \ref{thm:ell-2} in this section. Before proving the main results of this section, first, we set some notations and prove two useful lemmas. For any positive integer $\El$, we write $\xi\vdash\El$ to denote that $\xi$ is a partition of $\El$. Given a partition $\xi$ of the integer $\El$ (here we are allowing $\El$ to take value $0$), throughout this section $\xi_1$ denotes the largest part of $\xi$. In particular if $\xi\vdash 0$ then $f^{\xi}=1$ (as there is a unique Young diagram with zero boxes) and we set $\xi_1=0$.
\begin{thm}[{Plancherel formula, \cite[Theorem 4.1]{D1}}]
	Let $f_1$ and $f_2$ be two functions on the finite group G. Then
	\[\sum_{g\in G}f_1(g^{-1})f_2(g)=\frac{1}{|G|}\sum_{\rho\in \widehat{G}}d_{\rho}\Tr\left(\hat{f_1}(\rho)\hat{f_2}(\rho)\right),\]
	where the sum is over all irreducible representations  $\rho$ of $G$ and $d_{\rho}$ is the dimension of $\rho$.
\end{thm}
Recall that $U_G$ is the uniform distribution on the group $G$. Then using Lemma \ref{lem:only G-action} we have the following
\[\widehat{U}_G(\rho)=\begin{cases}
1&\text{ if }\rho=\trivial,\\
0&\text{ if }\rho\neq\trivial,
\end{cases}\quad\text{ for }\rho\in\widehat{G}.\]
Moreover, given any probability measure $p$ on the finite group $G$, we have $\widehat{p}(\trivial)=1$. Therefore setting $f_1=f_2=p^{*k}-U_G$, we have the following
\begin{equation}\label{eq:ell-2_spectral_relation}
p(x)=p(x^{-1})\text{ for all }x\in G\implies \left|\left|p^{*k}-U_{G}\right|\right|^2_{2}=\sum\limits_{\rho\in\widehat{G}\setminus\{\trivial\}}d_{\rho}\Tr\left(\left(\widehat{p}(\rho)\right)^{2k}\right).
\end{equation}
We now state the Diaconis-Shahshahani upper bound lemma. The proof follows from the Cauchy-Schwarz inequality and \eqref{eq:ell-2_spectral_relation}.
\begin{lem}[{\cite[Lemma 4.2]{D1}}]\label{Upper Bound Lemma}
	Let $p$ be a probability measure on a finite group $G$ such that $p(x)=p(x^{-1})$ for all $x\in G$. Suppose the random walk on $G$ driven by $p$ is irreducible. Then we have the following
	\[\left|\left|p^{*k}-U_{G}\right|\right|^2_{\emph{TV}}\leq\frac{1}{4}\left|\left|p^{*k}-U_{G}\right|\right|^2_{2}=\frac{1}{4}\sum\limits_{\rho\in\widehat{G}\setminus\{\trivial\}}d_{\rho}\Tr\left(\left(\widehat{p}(\rho)\right)^{2k}\right),\]
	where the sum is over all non-trivial irreducible representations  $\rho$ of $G$ and $d_{\rho}$ is the dimension of $\rho$. 
\end{lem}
\begin{defn}
	Let $A$ be a non empty set. Then the \emph{indicator function} of $A$ is denoted by $\Indf_A$ and is defined by
	\[\Indf_A(x)=\begin{cases}
	1&\text{ if }x\in A\\
	0&\text{ if }x\notin A.
	\end{cases}\]
\end{defn}
\begin{lem}\label{lem:UB_1}
	Let $\El$ be a positive integer and $s$ be any non-negative real number. Then we have
	\[\sum_{\lambda\vdash\El}(f^{\lambda})^2\left(\frac{\lambda_1-s}{\El}\right)^{2k}<e^{-\frac{2ks}{\El}}e^{\El^2e^{-\frac{2k}{\El}}}.\]
\end{lem}
\begin{proof}
	For $\zeta\vdash\left(\El-\lambda_1\right)$, recall that $\zeta_1$ denotes the largest part of $\zeta$. Since $\zeta_1\leq\lambda_1$ implies $f^{\lambda}\leq\binom{\El}{\lambda_1}f^{\zeta}$. Therefore $\displaystyle\sum_{\lambda\vdash\El}(f^{\lambda})^2\left(\frac{\lambda_1-s}{\El}\right)^{2k}$ is less than or equal to 
	\begin{align}\label{eq:UB_1.1}
		\sum_{\lambda_1=1}^{\El}\sum_{\substack{\zeta\vdash(\El-\lambda_1)\\\zeta_1\leq\lambda_1}}\binom{\El}{\lambda_1}^2(f^{\zeta})^2\left(\frac{\lambda_1-s}{\El}\right)^{2k}&\leq\sum_{\lambda_1=1}^{\El}\binom{\El}{\lambda_1}^2\left(\frac{\lambda_1-s}{\El}\right)^{2k}\sum_{\zeta\vdash(\El-\lambda_1)}(f^{\zeta})^2\nonumber\\
		&=\sum_{u=0}^{\El-1}\binom{\El}{u}^2\left(1-\frac{u+s}{\El}\right)^{2k}u!.
	\end{align}
	Equality in \eqref{eq:UB_1.1} is obtained by using $\displaystyle\sum_{\zeta\vdash(N-\lambda_1)}(f^{\zeta})^2=(N-\lambda_1)!$, and writing $u=\El-\lambda_1$. Using $1-x\leq e^{-x}$ for all $x\geq 0$ and $\binom{\El}{u}\leq\frac{\El^u}{u!}$, the expression in the right hand side of \eqref{eq:UB_1.1} is less than or equal to
	\[\sum_{u=0}^{\El-1}\frac{\El^{2u}}{u!}e^{-\frac{2k}{\El}(u+s)}<e^{-\frac{2ks}{\El}}\sum_{u=0}^{\infty}\frac{1}{u!}\left(\El^2e^{-\frac{2k}{\El}}\right)^u=e^{-\frac{2ks}{\El}}e^{\El^2e^{-\frac{2k}{\El}}}.\qedhere\]
\end{proof}
An immediate corollary of Lemma \ref{lem:UB_1} follows from the fact
\[\left(f^{\lambda}\right)^2\left(\frac{\lambda_1-s}{\El}\right)^{2k}=\left(\frac{\El-s}{\El}\right)^{2k},\text{ if }\lambda=\left(\El\right)\vdash\El.\]
\begin{cor}\label{rem:UB_1}
	Following the notations of Lemma \ref{lem:UB_1}, we have
	\[\sum_{\substack{\lambda\vdash\El\\\lambda\neq (\El)}}(f^{\lambda})^2\left(\frac{\lambda_1-s}{\El}\right)^{2k}<e^{-\frac{2ks}{\El}}e^{\El^2e^{-\frac{2k}{\El}}}-\left(\frac{\El-s}{\El}\right)^{2k}.\]
\end{cor}
\begin{lem}\label{lem:UB_2}
Let $\mu=(\mu^{(1)},\dots,\mu^{(t)})\in\ynirr$. Recall that $\mu^{(j)}_1\;($respectively $\mu^{(j)^{\prime}}_1)$ denotes the largest part of $\mu^{(j)}\;($respectively its conjugate$)$ for $1\leq j\leq t$. Then we have
\begin{align*}
&\sum_{T\in\tab_{\Gr}(n,\mu)}\left(\frac{c(b_T(n))+\langle\chi^{r_T(n)},\chi^{\trivial}\rangle}{n\dim(W^{r_T(n)})}\right)^{2k}\\
<\;\;&\binom{n}{m_1,\dots, m_t}f^{\mu^{(1)}}\cdots f^{\mu^{(t)}}\sum_{j=1}^{t}\left(\M_j^{2k}+\M'^{2k}_j\right)\Indf_{(0,\infty)}(m_j),
\end{align*}
where $\M_j:=\frac{\mu^{(j)}_1-1+\langle\chi^{\sigma_j},\chi^{\trivial}\rangle}{nd_j}$ and $\M'_j:=\frac{\mu^{(j)^{\prime}}_1-1+\langle\chi^{\sigma_j},\chi^{\trivial}\rangle}{nd_j}$ for each $1\leq j\leq t$.
\end{lem}
\begin{proof}
	Let $\mathcal{T}_i=\{(T_1,\dots,T_t)\in\tab_{\Gr}(n,\mu)\mid b_T(n)\text{ is in }T_i\}\text{ for each }1\leq i\leq t$. Then $\tab_{\Gr}(n,\mu)$ is the disjoint union of the sets $\mathcal{T}_1,\dots,\mathcal{T}_t$. Therefore we have
	\[\displaystyle\sum_{T\in\tab_{\Gr}(n,\mu)}\left(\frac{c(b_T(n))+\langle\chi^{r_T(n)},\chi^{\trivial}\rangle}{n\dim(W^{r_T(n)})}\right)^{2k}=\displaystyle\sum_{i=1}^{t}\displaystyle\sum_{T\in\mathcal{T}_i}\left(\frac{c(b_T(n))+\langle\chi^{\sigma_i},\chi^{\trivial}\rangle}{nd_i}\right)^{2k}\Indf_{(0,\infty)}(m_i)\]
	and this is equal to,
	\begin{align}\label{eq:UB_2.1}
	&\sum_{i=1}^{t}\binom{n-1}{m_1,..,m_i-1,..,m_t}\frac{f^{\mu^{(1)}}\cdots f^{\mu^{(t)}}}{f^{\mu^{(i)}}}\sum_{T_i\in\tab(\mu^{(i)})}\left(\frac{c(b_{T_i}(m_i))+\langle\chi^{\sigma_i},\chi^{\trivial}\rangle}{nd_i}\right)^{2k}\Indf_{(0,\infty)}(m_i)\nonumber\\
	<&\;\sum_{i=1}^{t}\binom{n}{m_1,\dots,m_t}\frac{f^{\mu^{(1)}}\cdots f^{\mu^{(t)}}}{f^{\mu^{(i)}}}\sum_{T_i\in\tab(\mu^{(i)})}\left(\M_i^{2k}+\M'^{2k}_i\right)\Indf_{(0,\infty)}(m_i).
	\end{align}
	The inequality in \eqref{eq:UB_2.1} holds because $T_i\in\tab(\mu^{(i)})$ implies the following:
	\begin{align*}
	&\left(\frac{c(b_{T_i}(m_i))+\langle\chi^{\sigma_i},\chi^{\trivial}\rangle}{nd_i}\right)^{2k}\\
	\leq&\max\Bigg\{\left(\frac{\mu^{(i)}_1-1+\langle\chi^{\sigma_i},\chi^{\trivial}\rangle}{nd_i}\right)^{2k},\left(\frac{\mu^{(i)^{\prime}}_1-1-\langle\chi^{\sigma_i},\chi^{\trivial}\rangle}{nd_i}\right)^{2k}\Bigg\}\\
	\leq&\max\Bigg\{\left(\frac{\mu^{(i)}_1-1+\langle\chi^{\sigma_i},\chi^{\trivial}\rangle}{nd_i}\right)^{2k},\left(\frac{\mu^{(i)^{\prime}}_1-1+\langle\chi^{\sigma_i},\chi^{\trivial}\rangle}{nd_i}\right)^{2k}\Bigg\},\quad\text{ as }\langle\chi^{\sigma_i},\chi^{\trivial}\rangle=0\text{ or }1\\
	<&\left(\frac{\mu^{(i)}_1-1+\langle\chi^{\sigma_i},\chi^{\trivial}\rangle}{nd_i}\right)^{2k}+\left(\frac{\mu^{(i)^{\prime}}_1-1+\langle\chi^{\sigma_i},\chi^{\trivial}\rangle}{nd_i}\right)^{2k}=\M_i^{2k}+\M_i'^{2k}.
	\end{align*}
	Therefore the result follows from \eqref{eq:UB_2.1} and
	\[\displaystyle\sum_{T_i\in\tab(\mu^{(i)})}\left(\M_i^{2k}+\M'^{2k}_i\right)=f^{\mu^{(i)}}\left(\M_i^{2k}+\M'^{2k}_i\right).\qedhere\]
\end{proof}
\begin{prop}\label{prop:key_ineq}
	For the warp-transpose top with random shuffle on $\Gn$, we have 
	\begin{align*}
		4\;&\left|\left|P^{*k}-U_{\Gn}\right|\right|^2_{\emph{TV}}\leq\left|\left|P^{*k}-U_{\Gn}\right|\right|^2_{2}\;<\;2\left(e^{n^2e^{-\frac{2k}{n}}}-1\right)+e^{-\frac{4k}{n}}\\
		&+2e^{n^2e^{-\frac{2k}{n}}}\left(e^{n^2\left(|\Gr|-1\right)e^{-\frac{2k}{n}}}-1\right)+2(|\irrG|-1)n^2e^{-\frac{2k}{n}}e^{n^2e^{-\frac{2k}{n}}}\left(\frac{1}{n^2}+e^{n^2\left(|\Gr|-1\right)e^{-\frac{2k}{n}}}-1\right),
	\end{align*}
	for all $k\geq\max\{n,\;n\log n\}$.
\end{prop}
\begin{proof}
Let us recall that $\irrG=\{\sigma_1,\dots,\sigma_t\}\text{ and }\sigma_1=\trivial$, the trivial representation of $\Gr$. Given $\mu\in\ynirr$, throughout this proof we write $\mu=(\mu^{(1)},\dots,\mu^{(t)})$, where $\mu^{(i)}=\mu(\sigma_i)$, $\mu^{(i)}\vdash m_i$, and $\sum_{i=1}^{t}m_i=n$. Now using Lemma \ref{Upper Bound Lemma}, we have
\begin{equation}\label{eq:key_ineq_1}
	4\;\left|\left|P^{*k}-U_{\Gn}\right|\right|^2_{\text{TV}}\leq\left|\left|P^{*k}-U_{\Gn}\right|\right|^2_{2}=\sum_{\mu\in\ynirr : \;\mu(\trivial)\neq(n)}\dimension(V^{\mu})\Tr\left(\left(\widehat{P}(R)\big|_{V^{\mu}}\right)^{2k}\right).
\end{equation}
First we partition the set $\ynirr$ into two disjoint subsets $\mathcal{A}_1\text{ and }\mathcal{A}_2$ as follows:
\begin{align*}
	\mathcal{A}_1=\underset{1\leq i\leq t}{\cup}\mathcal{B}_i,\text{ where }
	\mathcal{B}_i&=\{\mu\in\ynirr\mid m_i=n,\; m_k=0\text{ for all }k\in[t]\setminus\{i\}\}\\
	\mathcal{A}_2&=\{\mu\in\ynirr\mid \sum_{k=1}^{t}m_k=n,\;0\leq m_k\leq n-1\}.
\end{align*}
It can be easily seen that $\mathcal{B}_i$'s are disjoint. Therefore by using Theorem \ref{thm: eigenvalues on irre}, Remark \ref{rem:all eigenvalues}, and $\sigma_1=\trivial$, the inequality \eqref{eq:key_ineq_1} become 
\begin{align}\label{eq:key_ineq_2}
	4\;\left|\left|P^{*k}-U_{\Gn}\right|\right|_{\text{TV}}^2\leq\left|\left|P^{*k}-U_{\Gn}\right|\right|^2_{2}=&\sum_{\substack{\mu\in \mathcal{B}_1\\\mu(\trivial)\neq(n)}}\dimension(V^{\mu})\sum_{T\in\tab_{\Gr}(n,\mu)}\left(\frac{c(b_T(n))+1}{nd_1}\right)^{2k}d_1^{n}\nonumber\\
	&+\sum_{i=2}^{t}\sum_{\mu\in \mathcal{B}_i}\dimension(V^{\mu})\sum_{T\in\tab_{\Gr}(n,\mu)}\left(\frac{c(b_T(n))}{nd_i}\right)^{2k}d_i^{n}\\
	&\hspace*{-3.5cm}+\sum_{\mu\in \mathcal{A}_2}\dimension(V^{\mu})\sum_{T\in\tab_{\Gr}(n,\mu)}\left(\frac{c(b_T(n))+\langle\chi^{r_T(n)},\chi^{\trivial}\rangle}{n\dimension(W^{r_T(n)})}\right)^{2k}d_1^{m_1}\cdots d_t^{m_t}.\nonumber
\end{align}
The sum of the first two terms in the right hand side of \eqref{eq:key_ineq_2} are equal to
\begin{align}\label{eq:key_ineq_2.1_0}
&\sum_{\substack{\lambda\vdash n\\\lambda_1\neq n}}f^{\lambda}d_1^n\sum_{T\in\tab(\lambda)}\left(\frac{c(b_T(n))+1}{nd_1}\right)^{2k}d_1^{n}+\sum_{i=2}^{t}\sum_{\lambda\vdash n}f^{\lambda}d_i^n\sum_{T\in\tab(\lambda)}\left(\frac{c(b_T(n))}{nd_i}\right)^{2k}d_i^{n}\nonumber\\
=&\;\frac{d_1^{2n}}{d_1^{2k}}\hspace*{-1ex}\sum_{\substack{\lambda\vdash n\\\lambda\neq(n),(1^n)}}\hspace*{-1.5ex}f^{\lambda}\sum_{T\in\tab(\lambda)}\left(\frac{c(b_T(n))+1}{n}\right)^{2k}+\left(\frac{n-2}{n}\right)^{2k}+\sum_{i=2}^{t}\frac{d_i^{2n}}{d_i^{2k}}\sum_{\lambda\vdash n}f^{\lambda}\sum_{T\in\tab(\lambda)}\left(\frac{c(b_T(n))}{n}\right)^{2k}\hspace*{-1ex}.
\end{align}
Now recalling $\lambda_1\;($respectively $\lambda^{\prime}_1)$ is the largest part of $\lambda\;($respectively its conjugate$)$, we have the following:
\begin{align*}
\left(\frac{c(b_T(n))+x}{n}\right)^{2k}\leq&\;\max\Bigg\{\left(\frac{\lambda_1-1+x}{n}\right)^{2k},\left(\frac{\lambda^{\prime}_1-1-x}{n}\right)^{2k}\Bigg\},\\
<&\;\left(\frac{\lambda_1-1+x}{n}\right)^{2k}+\left(\frac{\lambda^{\prime}_1-1+x}{n}\right)^{2k},\;\;\text{ for }T\in\tab(\lambda)\text{ and }x\geq0.
\end{align*}
This implies
\begin{align*}
\sum_{\substack{\lambda\vdash n\\\lambda\neq(n),(1^n)}}\hspace*{-1.5ex}f^{\lambda}\sum_{T\in\tab(\lambda)}\left(\frac{c(b_T(n))+1}{n}\right)^{2k}&<\sum_{\substack{\lambda\vdash n\\\lambda\neq(n),(1^n)}}\left(f^{\lambda}\right)^2\left(\left(\frac{\lambda_1}{n}\right)^{2k}+\left(\frac{\lambda^{\prime}_1}{n}\right)^{2k}\right)\\
&=\;2\sum_{\substack{\lambda\vdash n\\\lambda\neq(n),(1^n)}}\left(f^{\lambda}\right)^2\left(\frac{\lambda_1}{n}\right)^{2k},\quad\text{ using }f^{\lambda}=f^{\lambda^{\prime}}\\
\text{ and }\quad\sum_{\lambda\vdash n}f^{\lambda}\sum_{T\in\tab(\lambda)}\left(\frac{c(b_T(n))}{n}\right)^{2k}&<\sum_{\lambda\vdash n}\left(f^{\lambda}\right)^2\left(\left(\frac{\lambda_1-1}{n}\right)^{2k}+\left(\frac{\lambda^{\prime}_1-1}{n}\right)^{2k}\right)\\
&=\;2\sum_{\lambda\vdash n}\left(f^{\lambda}\right)^2\left(\frac{\lambda_1-1}{n}\right)^{2k},\quad\text{ using }f^{\lambda}=f^{\lambda^{\prime}}.
\end{align*}
Thus using $1-x\leq e^{-x}$ for $x\geq 0$, $k\geq n$, and $d_i\geq 1$ for all $1\leq i\leq t$, the expression in \eqref{eq:key_ineq_2.1_0} is bounded above by
\begin{align}\label{eq:key_ineq_2.1}
&\;2\sum_{\substack{\lambda\vdash n\\\lambda\neq (n)}}(f^{\lambda})^2\left(\frac{\lambda_1}{n}\right)^{2k}+\left(1-\frac{2}{n}\right)^{2k}+2\sum_{i=2}^{t}\sum_{\lambda\vdash n}(f^{\lambda})^2\left(\frac{\lambda_1-1}{n}\right)^{2k}\nonumber\\
< & \; 2\left(e^{n^2e^{-\frac{2k}{n}}}-1\right)+e^{-\frac{4k}{n}}+2(t-1)e^{-\frac{2k}{n}}e^{n^2e^{-\frac{2k}{n}}}.
\end{align}
The inequality in \eqref{eq:key_ineq_2.1} follows from Corollary \ref{rem:UB_1} and Lemma \ref{lem:UB_1}. Now recalling $\M_j:=\frac{\mu^{(j)}_1-1+\langle\chi^{\sigma_j},\chi^{\trivial}\rangle}{nd_j},\;\M'_j:=\frac{\mu^{(j)^{\prime}}_1-1+\langle\chi^{\sigma_j},\chi^{\trivial}\rangle}{nd_j}$, and using Lemma \ref{lem:UB_2}, the third term in the right hand side of \eqref{eq:key_ineq_2} is less than
\begin{align}\label{eq:key_ineq_2.2}
\sum_{\mu\in \mathcal{A}_2}\binom{n}{m_1,\dots, m_t}^2(f^{\mu^{(1)}})^2\cdots (f^{\mu^{(t)}})^2d_1^{2m_1}\dots d_t^{2m_t}\sum_{j=1}^{t}\left(\M_j^{2k}+\M'^{2k}_j\right)\Indf_{(0,\infty)}(m_j).
\end{align}
We now deal with \eqref{eq:key_ineq_2.2} by considering two separate cases namely $j=1$ and $1<j\leq t$. Now using
\[\sum_{\mu^{(1)}\vdash m_1}\left(f^{\mu^{(1)}}\right)^2\left(\frac{\mu^{(1)^{\prime}}_1}{nd_1}\right)^{2k}=\sum_{\mu^{(1)}\vdash m_1}\left(f^{\mu^{(1)}}\right)^2\left(\frac{\mu^{(1)}_1}{nd_1}\right)^{2k},\]
the partial sum corresponding to $j=1$ in \eqref{eq:key_ineq_2.2} is equal to,
\begin{equation}\label{eq:key_ineq_j_1.1-0}
	\sum_{m_1=1}^{n-1}\sum_{\substack{(m_2,\dots,m_t)\\\sum m_k=n-m_1\\\\0\leq m_k\leq n-1}}2\sum_{\substack{\mu^{(i)}\vdash m_i\\1\leq i\leq t}}\binom{n}{m_1}^2\binom{n-m_1}{m_2,\dots,m_t}^2(f^{\mu^{(1)}})^2\cdots (f^{\mu^{(t)}})^2d_1^{2m_1}\dots d_t^{2m_t}\left(\frac{\mu^{(1)}_1}{nd_1}\right)^{2k}
\end{equation}
Using $\displaystyle\sum_{\mu^{(i)}\vdash m_i}\left(f^{\mu^{(i)}}\right)^2=m_i!$ for $2\leq i\leq t,\;\binom{n-m_1}{m_2,\dots,m_t}=\frac{(n-m_1)!}{m_2!\cdots m_t!}$, and the multinomial theorem
\[\sum_{\substack{(m_2,\dots,m_t)\\\sum m_k=n-m_1\\\\0\leq m_k\leq n-1}}\binom{n-m_1}{m_2,\dots,m_t}(d_2^2)^{m_2}\dots (d_t^2)^{m_t}=(d_2^2+\cdots+d_t^2)^{n-m_1},\]
the expression in \eqref{eq:key_ineq_j_1.1-0} can be written as
\begin{align}\label{eq:key_ineq_j_1.1}
&2\sum_{m_1=1}^{n-1}(d_2^2+\cdots+d_t^2)^{n-m_1}\binom{n}{m_1}^2(n-m_1)!\left(\frac{1}{d_1}\right)^{2k-2m_1}\left(\frac{m_1}{n}\right)^{2k}\hspace*{-1.5ex}\sum_{\mu^{(1)}\vdash m_1}(f^{\mu^{(1)}})^2\left(\frac{\mu^{(1)}_1}{m_1}\right)^{2k}\nonumber\\
<&\;2\sum_{m_1=1}^{n-1}(d_2^2+\cdots+d_t^2)^{n-m_1}\binom{n}{m_1}^2(n-m_1)!\left(\frac{1}{d_1}\right)^{2k-2m_1}\left(\frac{m_1}{n}\right)^{2k}e^{m_1^2e^{-\frac{2k}{m_1}}}.
\end{align}
The inequality in \eqref{eq:key_ineq_j_1.1} follows from Lemma \ref{lem:UB_1}. As $n\geq m_1$, we have 
\[k\geq m_1\log m_1+\frac{m_1}{n}k-m_1\log n\implies m_1^2e^{-\frac{2k}{m_1}}\leq n^2e^{-\frac{2k}{n}}.\]
Thus writing $n-m_1$ by $u$, the expression in \eqref{eq:key_ineq_j_1.1} is less than or equal to
\begin{equation}\label{eq:key_ineq_j_1.2_0}
2e^{n^2e^{-\frac{2k}{n}}}\;\sum_{u=1}^{n-1}\left(\frac{d_2^2+\cdots+d_t^2}{d_1^2}\right)^u\left(\frac{1}{d_1}\right)^{2k-2n}\binom{n}{u}^2u!\left(1-\frac{u}{n}\right)^{2k}
\end{equation}
Now using $1-x\leq e^{-x}$ for all $x\geq 0$ and $d_1=1$ the expression in \eqref{eq:key_ineq_j_1.2_0} is less than or equal to
\begin{equation}\label{eq:key_ineq_j_1.2}
	2e^{n^2e^{-\frac{2k}{n}}}\;\sum_{u=1}^{n-1}\frac{1}{u!}\left(n^2\left(\frac{|\Gr|}{d_1^2}-1\right)e^{-\frac{2k}{n}}\right)^u\;<2e^{n^2e^{-\frac{2k}{n}}}\left(e^{\left(n^2\left(\frac{|\Gr|}{d_1^2}-1\right)e^{-\frac{2k}{n}}\right)}-1\right).
\end{equation}
Now using the notation $m_1,..,\widehat{m_j},..,m_t$ to denote $m_1,\dots,m_{j-1},m_{j+1},\dots,m_t$, and
\[\sum_{\mu^{(j)}\vdash m_j}\left(f^{\mu^{(j)}}\right)^2\left(\frac{\mu^{(j)^{\prime}}_1}{nd_j}\right)^{2k}=\sum_{\mu^{(j)}\vdash m_j}\left(f^{\mu^{(j)}}\right)^2\left(\frac{\mu^{(j)}_1}{nd_j}\right)^{2k},\]
the partial sum corresponding to $1<j\leq t$ in \eqref{eq:key_ineq_2.2} turns out to be
\begin{equation}\label{eq:key_ineq_j_1.3_0}
\sum_{m_j=1}^{n-1}\sum_{\substack{(m_1,\dots,\widehat{m_j},\dots,m_t)\\\sum m_k=n-m_j\\\\0\leq m_k\leq n-1}}\hspace*{-1ex}2\hspace*{-0.5ex}\sum_{\substack{\mu^{(i)}\vdash m_i\\1\leq i\leq t}}\binom{n}{m_j}^2\binom{n-m_j}{m_1,\dots,\widehat{m_j},\dots,m_t}^2(f^{\mu^{(1)}})^2\cdots(f^{\mu^{(t)}})^2d_1^{2m_1}\dots d_t^{2m_t}\zeta^{2k},
\end{equation}
where $\zeta=\frac{\mu^{(j)}_1-1}{nd_j}$. Using $\displaystyle\sum_{\mu^{(i)}\vdash m_i}\left(f^{\mu^{(i)}}\right)^2=m_i!$ for $i\in\{1,\dots,t\}\setminus \{j\},\;\binom{n-m_j}{m_1,\dots,\widehat{m_j}\dots,m_t}=\frac{(n-m_j)!}{\displaystyle\prod_{k\neq j}m_k!}$, and the multinomial theorem
	\[\sum_{\substack{(m_1,\dots,\widehat{m_j},\dots,m_t)\\\sum_{k\neq j} m_k=n-m_j\\\\0\leq m_k\leq n-1}}\hspace*{-1.5ex}\binom{n-m_j}{m_1,\dots,\widehat{m_j},\dots,m_t}(d_1^2)^{m_1}\dots(d_{j-1}^2)^{m_{j-1}} (d_{j+1}^2)^{m_{j+1}}\dots (d_t^2)^{m_t}=\left(\displaystyle\sum_{k\neq j}d_k^2\right)^{n-m_j},\]
the expression given in \eqref{eq:key_ineq_j_1.3_0} is equal to the following
\begin{equation*}
2\sum_{m_j=1}^{n-1}(d_1^2+\cdots+d_t^2-d_j^2)^{n-m_j}\binom{n}{m_j}^2(n-m_j)!\left(\frac{1}{d_j}\right)^{2k-2m_j}\left(\frac{m_j}{n}\right)^{2k}\sum_{\mu^{(j)}\vdash m_j}(f^{\mu^{(j)}})^2\left(\frac{\mu^{(j)}_1-1}{m_j}\right)^{2k}
\end{equation*}
\begin{equation}\label{eq:key_ineq_j_1.3}
<\;2\sum_{m_j=1}^{n-1}(d_1^2+\cdots+d_t^2-d_j^2)^{n-m_j}\binom{n}{m_j}^2(n-m_j)!\left(\frac{1}{d_j}\right)^{2k-2m_j}\left(\frac{m_j}{n}\right)^{2k}e^{-\frac{2k}{m_j}}e^{m_j^2e^{-\frac{2k}{m_j}}}.
\end{equation}
The inequality in \eqref{eq:key_ineq_j_1.3} follows from Lemma \ref{lem:UB_1}. As $n\geq m_j$, we have 
\[k\geq m_j\log m_j+\frac{m_j}{n}k-m_j\log n\implies m_j^2e^{-\frac{2k}{m_j}}\leq n^2e^{-\frac{2k}{n}}.\]
Thus writing $n-m_j$ by $v$ and using $\frac{1}{m_j}\leq 1$, the expression in \eqref{eq:key_ineq_j_1.3} is less than or equal to 
\begin{equation}\label{eq:key_ineq_j>1_0}
2n^2e^{-\frac{2k}{n}}e^{n^2e^{-\frac{2k}{n}}}\;\sum_{v=1}^{n-1}\left(\frac{d_1^2+\cdots+d_t^2-d_j^2}{d_j^2}\right)^v\left(\frac{1}{d_j}\right)^{2k-2n}\binom{n}{v}^2v!\left(1-\frac{v}{n}\right)^{2k}
\end{equation}
Now using $1-x\leq e^{-x}$ for all $x\geq 0$ and $d_j^{2k-2n}\geq 1$ for all $j\in\{1,\dots,t\}$, the expression in \eqref{eq:key_ineq_j>1_0} is less than or equal to 
\begin{align}\label{eq:key_ineq_j>1}
&2n^2e^{-\frac{2k}{n}}e^{n^2e^{-\frac{2k}{n}}}\;\sum_{v=1}^{n-1}\frac{1}{v!}\left(n^2\left(\frac{|\Gr|}{d_j^2}-1\right)e^{-\frac{2k}{n}}\right)^v\nonumber\\
<&\;2n^2e^{-\frac{2k}{n}}e^{n^2e^{-\frac{2k}{n}}}\left(e^{\left(n^2\left(\frac{|\Gr|}{d_j^2}-1\right)e^{-\frac{2k}{n}}\right)}-1\right).
\end{align}
Therefore the proposition follows from \eqref{eq:key_ineq_2}, \eqref{eq:key_ineq_2.1}, \eqref{eq:key_ineq_j_1.2}, \eqref{eq:key_ineq_j>1} and $\frac{1}{d_j}\leq 1$ for all $1\leq j\leq t$.
\end{proof}
\begin{thm}\label{lem:G_n Upper Bound}
	For the random walk on $\Gn$ driven by $P$ we have the following:
	\begin{enumerate}
	\item Let $C>0$. If $k\geq n\log n+\frac{1}{2}n\log(|\Gr|-1)+Cn$, then
	\[\left|\left|P^{*k}-U_{\Gn}\right|\right|_{\emph{TV}}\leq\frac{1}{2}\left|\left|P^{*k}-U_{\Gn}\right|\right|_{2}<\sqrt{2}\left(e^{-2C}+1\right)e^{-C}+o(1).\]
	\item For any $\epsilon\in (0,1)$, if we set $k_n=\big\lfloor(1+\epsilon)\left(n\log n+\frac{1}{2}n\log\left(|\Gr|-1\right)\right)\big\rfloor$, then
	\[\lim_{n\rightarrow\infty}\left|\left|P^{*k_n}-U_{\Gn}\right|\right|_{2}=0.\]
	\end{enumerate}
\end{thm}
\begin{proof}
	Using $k\geq n\log n+\frac{1}{2}n\log(|\Gr|-1)+Cn$ and Proposition \ref{prop:key_ineq} we have the following:
	\begin{align}
		4\;\left|\left|P^{*k}-U_{\Gn}\right|\right|_{\text{TV}}^2&\leq\left|\left|P^{*k}-U_{\Gn}\right|\right|^2_{2}\nonumber\\
		&<2\left(e^{\frac{e^{-2C}}{|\Gr|-1}}-1\right)+\frac{e^{-4C}}{n^4\left(|\Gr|-1\right)^2}+2e^\frac{e^{-2C}}{|\Gr|-1}\left(e^{e^{-2C}}-1\right)\label{eq:Upper_Bound_pf_0}\\
		&\hspace*{1cm}+2(|\irrG|-1)\times\frac{e^{-2C}}{|\Gr|-1}\times e^{\frac{e^{-2C}}{|\Gr|-1}}\left(\frac{1}{n^2}+e^{e^{-2C}}-1\right).\nonumber
	\end{align}
	The sequence $\{\Gr\}_1^{\infty}$ consists of non-trivial finite groups, thus $\frac{1}{|\Gr|-1}\leq 1$. Also, $\frac{|\irrG|-1}{|\Gr|-1}\leq 1$. Therefore the expression in the right hand side of \eqref{eq:Upper_Bound_pf_0} is less than
	\[\left(2+2e^{e^{-2C}}+2e^{-2C}e^{e^{-2C}}\right)\left(e^{e^{-2C}}-1\right)+\frac{e^{-4C}}{n^4}+\frac{2e^{-2C}e^{e^{-2C}}}{n^2}\]
	Now using $e^x-1<2x$ for $0<x\leq 1$ we have
	\begin{align*}\label{eq:Upper_Bound_pf_1}
		4\;\left|\left|P^{*k}-U_{\Gn}\right|\right|_{\text{TV}}^2\leq\left|\left|P^{*k}-U_{\Gn}\right|\right|^2_{2}&<\left(4+6e^{-2C}+4e^{-4C}\right)2e^{-2C}+o(1)\\
		\implies\;\left|\left|P^{*k}-U_{\Gn}\right|\right|_{\text{TV}}\leq\frac{1}{2}\left|\left|P^{*k}-U_{\Gn}\right|\right|_{2}&<\left(\sqrt{2+3e^{-2C}+2e^{-4C}}\right) e^{-C}+o(1)\\
		&<\sqrt{2}\left(e^{-2C}+1\right)e^{-C}+o(1).
	\end{align*}
This proves the first part of the theorem.

For any $\epsilon\in(0,1)$, setting $k_n=\left\lfloor(1+\epsilon)\left(n\log n+\frac{1}{2}n\log\left(|\Gr|-1\right)\right)\right\rfloor$, we have 
\[k_n+1\geq(1+\epsilon)\left(n\log n+\frac{1}{2}n\log\left(|\Gr|-1\right)\right)\implies e^{-\frac{k_n}{n}}\leq\frac{e^{\frac{1}{n}}}{n^{1+\epsilon}(|\Gr|-1)^{\frac{1+\epsilon}{2}}}.\]
Therefore Proposition \ref{prop:key_ineq} implies
\begin{align}\label{eq:cutoff-UB-proof}
	0&\leq4\;\left|\left|P^{*k_n}-U_{\Gn}\right|\right|_{\text{TV}}^2\leq\left|\left|P^{*k_n}-U_{\Gn}\right|\right|^2_{2}\nonumber\\
	&<2\left(e^{\frac{e^{\frac{2}{n}}}{n^{2\epsilon}(|\Gr|-1)^{1+\epsilon}}}-1\right)+\frac{e^{\frac{4}{n}}}{n^{4+4\epsilon}\left(|\Gr|-1\right)^{2+2\epsilon}}+2e^{\frac{e^{\frac{2}{n}}}{n^{2\epsilon}(|\Gr|-1)^{1+\epsilon}}}\left(e^{\frac{e^{\frac{2}{n}}}{n^{2\epsilon}(|\Gr|-1)^{\epsilon}}}-1\right)\\
	&\hspace*{1.75cm}+2(|\irrG|-1)\times \frac{e^{\frac{2}{n}}}{n^{2\epsilon}(|\Gr|-1)^{1+\epsilon}}\times e^{\frac{e^{\frac{2}{n}}}{n^{2\epsilon}(|\Gr|-1)^{1+\epsilon}}}\left(\frac{1}{n^2}+e^{\frac{e^{\frac{2}{n}}}{n^{2\epsilon}(|\Gr|-1)^{\epsilon}}}-1\right)\nonumber
\end{align}
The right hand side of \eqref{eq:cutoff-UB-proof} converges to zero as $n\rightarrow\infty$ because $\frac{1}{|\Gr|-1},\frac{|\irrG|-1}{|\Gr|-1}\leq 1$. Hence the second part follows.
\end{proof}
\begin{proof}[Proof of Theorem \ref{thm:mixing_main}]
	Let $\varepsilon>0$ and $\tau_{\text{mix}}^{(n)}(\varepsilon)$ (respectively $t_{\text{mix}}^{(n)}(\varepsilon)$) be the $\ell^2$-mixing time (respectively total variation mixing time) with tolerance level $\varepsilon$ for the warp-transpose top with random shuffle on $\Gn$. We choose $C_{\varepsilon}>0$ such that $\sqrt{2}\left(e^{-2C_{\varepsilon}}+1\right)e^{-C_{\varepsilon}}<\frac{\varepsilon}{4}$. Then the first part of Theorem \ref{lem:G_n Upper Bound} ensures the existence of positive integer $N$ such that the following hold for all $n\geq N$,
	\begin{align*}
		k\geq n\log n+\frac{1}{2}n\log(|\Gr|-1)+C_{\varepsilon}n
		\implies& \left|\left|P^{*k}-U_{\Gn}\right|\right|_{\text{TV}}\leq\frac{1}{2}\left|\left|P^{*k}-U_{\Gn}\right|\right|_2<\frac{\varepsilon}{2}\\
		\implies& \left|\left|P^{*k}-U_{\Gn}\right|\right|_{\text{TV}}<\varepsilon\text{ and }\left|\left|P^{*k}-U_{\Gn}\right|\right|_2<\varepsilon.
	\end{align*}
	Finally, using $n\log n+\frac{1}{2}n\log(|\Gr|-1)+C_{\varepsilon}n< \; 2\left(n\log n+\frac{1}{2}n\log(|\Gr|-1)\right)$ for all $n\geq N$, we can conclude that
	\begin{align*}
		&\tau_{\text{mix}}^{(n)}(\varepsilon)\leq 2\left(n\log n+\frac{1}{2}n\log(|\Gr|-1)\right)\text{ and }t_{\text{mix}}^{(n)}(\varepsilon)\leq 2\left(n\log n+\frac{1}{2}n\log(|\Gr|-1)\right).
	\end{align*}
	Thus the theorem follows.
\end{proof}
We now establish a lower bound of $\left|\left|P^{*k}-U_{\Gn}\right|\right|_2$ that will be useful in proving the $\ell^2$-cutoff.
\begin{prop}\label{prop:ell-2-LB-kay_inequality}
	For large $n$, we have 
	\[\left|\left|P^{*k}-U_{\Gn}\right|\right|_2>\sqrt{(n-2+n(|\Gr|-1))(n-1)}\;e^{-\frac{k}{n}}.\]
\end{prop}
\begin{proof}
	Recall that the irreducible representations of $\Gn$ are parameterised by the elements of $\ynirr$. We now use Theorem \ref{thm: eigenvalues on irre} to compute the eigenvalues of the restriction of $\widehat{P}(R)$ to some irreducible $\Gn$-modules. The eigenvalues of the restriction of $\widehat{P}(R)$ to the irreducible $\Gn$-module indexed by
	\begin{equation}\label{eq:ell-2_LB-1}
	\begin{split}
	&\hspace{1.725cm}n-1\\[-1.5ex]
	&\left(\begin{array}{c}
	\overbrace{\young(\;\;{\;$\cdots$\;\;}\;\;,\;)}
	\end{array},\emptyset,\dots,\emptyset\right)\in\ynirr
	\end{split}
	\end{equation}
	are given below.
	\begin{center}
		\begin{tabular}{ccc}
			\text{Eigenvalues:} & $1-\frac{1}{n}$ &\quad $0$\\
			\text{Multiplicities:}& $n-2$ &\quad $1$
		\end{tabular}
	\end{center}
	The eigenvalues of the restriction of $\widehat{P}(R)$ to the irreducible $\Gn$-modules indexed by Young $\Gr$-diagram with $n$ boxes of the following form
	\begin{equation}\label{eq:ell-2_LB-2}
	\begin{split}
	&\hspace{1.725cm}n-1\\[-2.5ex]
	&\left(\begin{array}{c}\overbrace{\young(\;\;{\;$\cdots$\;\;}\;\;)}\end{array},\emptyset,\dots,\emptyset,\underset{\uparrow}{\begin{array}{c}\yng(1)\end{array}},\emptyset,\dots,\emptyset\right)\in\ynirr,\;\text{ for }1<i\leq |\irrG|\\[-1ex]
	&\hspace{6cm}i\text{th position.}
	\end{split}
	\end{equation}
	are given below.
	\begin{center}
		\begin{tabular}{ccc}
			\text{Eigenvalues:} & $1-\frac{1}{n}$ &\quad $0$\\
			\text{Multiplicities:}& $(n-1)d_i$ &\quad $d_i$
		\end{tabular}
	\end{center}
	Now \eqref{eq:ell-2_spectral_relation} implies
	\begin{align}
	\left|\left|P^{*k}-U_{\Gn}\right|\right|_2^2&> (n-1)d_1^n\left((n-2)\left(1-\frac{1}{n}\right)^{2k}\right)+\sum_{i=2}^{|\irrG|}nd_1^{n-1}d_i\left((n-1)d_i\left(1-\frac{1}{n}\right)^{2k}\right)\nonumber\\
	&\approx (n-1)(n-2)e^{-\frac{2k}{n}}+n(n-1)(|\Gr|-1)e^{-\frac{2k}{n}}.\label{eq:ell-2-LB}
	\end{align}
	Here `$a_n\approx b_n$' means `$a_n$ is asymptotic to $b_n$' i.e. $\frac{a_n}{b_n}=1$ as $n\rightarrow\infty$. We have used $d_1=1$ and $\displaystyle\sum_{i=1}^{|\widehat{G}_n|}d_i^2=|\Gr|$ to obtain \eqref{eq:ell-2-LB}. Therefore \eqref{eq:ell-2-LB} implies
	\[\left|\left|P^{*k}-U_{\Gn}\right|\right|_2>\sqrt{(n-2+n(|\Gr|-1))(n-1)}\;e^{-\frac{k}{n}}\]
	for large $n$.
\end{proof}
\begin{proof}[Proof of Theorem \ref{thm:ell-2}]
	For any $\epsilon\in (0,1)$, the second part of Theorem \ref{lem:G_n Upper Bound} implies
	\begin{equation}\label{eq:ell-2-cutoff-proof-UB-side}
		\lim_{n\rightarrow\infty}\left|\left|P^{*\left\lfloor(1+\epsilon)\left(n\log n+\frac{1}{2}n\log\left(|\Gr|-1\right)\right)\right\rfloor}-U_{\Gn}\right|\right|_{2}=0.
	\end{equation}
	Again, Proposition \ref{prop:ell-2-LB-kay_inequality} implies the following
	\begin{equation}\label{eq:ell-2-cutoff-proof1}
		\left|\left|P^{*\left\lfloor(1-\epsilon)\left(n\log n+\frac{1}{2}n\log\left(|\Gr|-1\right)\right)\right\rfloor}-U_{\Gn}\right|\right|_{2}>\sqrt{\left(1+\frac{n-2}{n(|\Gr|-1)}\right)\left(1-\frac{1}{n}\right)}n^{\epsilon}(|\Gr|-1)^{\frac{\epsilon}{2}},
	\end{equation}
	for large $n$. The right hand side of the inequality \eqref{eq:ell-2-cutoff-proof1} tends to infinity as $n\rightarrow\infty$. Therefore from \eqref{eq:ell-2-cutoff-proof-UB-side} and \eqref{eq:ell-2-cutoff-proof1}, we can conclude that the warp-transpose top with random shuffle on $\Gn$ satisfies the $\ell^2$-cutoff phenomenon with cutoff time $n\log n+\frac{1}{2}n\log(|\Gr|-1)$. This completes the proof of Theorem \ref{thm:ell-2}.
\end{proof}
\section{Upper bound for total variation distance}\label{sec:upper bound}
	 In this section, we use the coupling argument to obtain an upper bound of $\left|\left|P^{*k}-U_{\Gn}\right|\right|_{\text{TV}}$ when $k\geq n\log n+Cn,\;C>1$. Our method uses the known upper bound for the transpose top with random shuffle \cite[Theorem 5.1]{D1} and coupling argument. Let us first recall Markov chain coupling.
\begin{defn}\label{def:coupling}
	A \emph{coupling} of Markov chains with transition matrix $M$ is a process $(X_t,Y_t)_t$ such that both $X:=(X_t)_t$ and $Y:=(Y_t)_t$ are Markov chains with transition matrix $M$ and with possibly different initial distribution.
\end{defn}
Given a coupling $(X_t,Y_t)_t$ of a Markov chain with transition matrix $M$, suppose that $T_{\text{couple}}$ is a random time such that
	\[X_t=Y_t\quad\text{for }t\geq T_{\text{couple}}.\]
Then for every pair of initial distribution $(\mu,\nu)$
\begin{equation}\label{eq:coupling-ineq}
	\left|\left|\mu M^t-\nu M^t\right|\right|_{\text{TV}}\leq\mathbb{P}(T_{\text{couple}}>t).
\end{equation}
The coupling is called \emph{successful} if 
	\[\mathbb{P}(T_{\text{couple}}<\infty)=1.\]
The random time $T_{\text{couple}}$ is known as the \emph{coupling time}. The coupling is called \emph{maximal} if equality holds in \eqref{eq:coupling-ineq}, i.e.
	\[\left|\left|\mu M^t-\nu M^t\right|\right|_{\text{TV}}=\mathbb{P}(T_{\text{couple}}>t).\]
\begin{thm}[{\cite[Theorem 4]{Griffeath}}]\label{thm:coupling}
	Any Markov chain $X$ has a maximal coupling $($achieving equality in \eqref{eq:coupling-ineq}$)$. Thus there exists a successful maximal coupling for $X$ if and only if $X$ is weakly ergodic $($i.e., irrespective of the initial distribution, the chain will converge to a unique distribution as the number of transition approaches infinity$)$.
	
	In particular, an irreducible and aperiodic random walk on finite group is weakly ergodic.
\end{thm}
Before going to the main theorem, we recall the coupon collector problem \cite[Section 2.2]{LPW} and prove a useful lemma. Suppose a shop sells $n$ different types of coupons. A collector visits the shop and buys coupons. Each coupon he buys is equally likely to be each of the $n$ types. The collector desires a complete set of all $n$ types. Let $\mathscr{T}$ be the minimum (random) number of coupons collected by the collector to have all the $n$ types; $\mathscr{T}$ is usually known as the \emph{coupon collector random variable}. Then we have the following \cite[Proposition 2.4]{LPW}
\begin{equation}\label{eq:coupon-collector}
	\mathbb{P}\left(\mathscr{T}>\lceil n\log n+C'n\rceil\right)\leq e^{-C'},\quad\text{ for any }C'>0.
\end{equation}
Let $\mathscr{T}_n$ be the minimum (random) number of coupons collected by the collector to have the coupon of type $n$. Then we define the twisted coupon collector random variable $\mathscr{T}'$ as follows:
\begin{equation}\label{eq:twisted-coupon-collector-time}
	\mathscr{T}'=\begin{cases}
		\mathscr{T},&\text{ if the last collected coupon is of type $n$},\\
		\mathscr{T}+\mathscr{T}_n,&\text{ if the last collected coupon is not of type $n$},
	\end{cases}
\end{equation}
i.e., $\mathscr{T}'$ is the minimum (random) number of coupons collected by the collector to collect every coupon at least once and the last collected coupon is of type $n$.
\begin{lem}\label{lem:twisted-coupon-collector-time}
	For the twisted coupon collector random variable defined in \eqref{eq:twisted-coupon-collector-time}, we have
	\[\mathbb{P}\left(\mathscr{T}'>\lceil n\log n+Cn\rceil\right)\leq (e+1)e^{-\frac{C}{2}},\quad\text{ for any }C>0.\]
\end{lem}
\begin{proof}
	The definition of the twisted coupon collector random variable $\mathscr{T}'$  implies that \[\mathscr{T}+\mathscr{T}_n\geq \mathscr{T}',\text{ where the distribution of }\mathscr{T}_n\text{ is geometric with parameter }\frac{1}{n}.\]
	Therefore,
	\begin{align*}
		&\mathbb{P}\left(\mathscr{T}'>\lceil n\log n+Cn\rceil\right)\\
		\leq&\mathbb{P}\left(\mathscr{T}+\mathscr{T}_n>\lceil n\log n+Cn\rceil\right)\\
		=&\mathbb{P}\left(\mathscr{T}+\mathscr{T}_n>\lceil n\log n+Cn\rceil,\mathscr{T}_n>\frac{C}{2}n-1\right)\\
		&\quad\quad\quad+\mathbb{P}\left(\mathscr{T}+\mathscr{T}_n>\lceil n\log n+Cn\rceil,\mathscr{T}_n\leq\frac{C}{2}n-1\right)\\
		\leq&\mathbb{P}\left(\mathscr{T}_n>\frac{C}{2}n-1\right)+\mathbb{P}\left(\mathscr{T}>\left\lceil n\log n+\frac{C}{2}n\right\rceil\right)
	\end{align*}
	\begin{align*}
		\leq&\sum_{k\geq\left\lfloor\frac{C}{2}n-1\right\rfloor+1}\left(1-\frac{1}{n}\right)^{k-1}\frac{1}{n}+e^{-\frac{C}{2}},\quad\text{by using \eqref{eq:coupon-collector} and }\mathscr{T}_n\sim\text{Geom}\left(\frac{1}{n}\right)\\
		=&\left(1-\frac{1}{n}\right)^{\left\lfloor\frac{C}{2}n-1\right\rfloor}+e^{-\frac{C}{2}}\leq(e+1)e^{-\frac{C}{2}}.\qedhere
	\end{align*}
\end{proof}
\begin{thm}\label{thm:TV-main_UB}
		For the random walk on $\Gn$ driven by $P$ we have the following:
		\begin{enumerate}
			\item Let $C>1$. If $k\geq n\log n+Cn$, then we have
			\[\left|\left|P^{*k}-U_{\Gn}\right|\right|_{\emph{TV}}<ae^{-C}+(e+1)e^{-\frac{C}{2}},\]
			where the constant $a$ is the universal constant given in \eqref{eq:TTR-coupling-UB}.
			\item For any $\epsilon\in (0,1)$,
			\[\lim_{n\rightarrow\infty}\left|\left|P^{*\lfloor(1+\epsilon)n\log n\rfloor}-U_{\Gn}\right|\right|_{\emph{TV}}=0.\]
		\end{enumerate}
\end{thm}
\begin{proof}
	We construct a coupling $(\mathscr{X}_k,\mathscr{Y}_k)_k$ for the warp-transpose top with random shuffle using a successful maximal coupling of the transpose top with random shuffle on $S_n$. Let $\mathscr{X}_0$ be the identity element $(\G1,\dots,\G1;\1)$ of $\Gn$, and $\mathscr{Y}_0$ be a random element of $\Gn$ (chosen uniformly). Thus the last element of $\mathscr{X}_0$ is the identity permutation $\1$, and the last element of $\mathscr{Y}_0$ is a random permutation (say) $\zeta$. 
	
	First we focus on the transpose top with random shuffle on $S_n$. Let us define the probability measure $\mathscr{P}$ on $S_n$ that generates the transpose top with random shuffle on $S_n$.
	\begin{equation}\label{eq:defn_of_script_P_a}
		\mathscr{P}(\pi)=\begin{cases}
			\frac{1}{n},&\text{ if }\pi=(i,n),\;1\leq i\leq n,\\
			0,&\text{ otherwise },
		\end{cases}\quad\quad\text{ for }\pi\in S_n,
	\end{equation}
	where $(i,n)$ denotes the transposition in $S_n$ interchanging $i$ and $n$; here we set $(n,n):=\1$. Recall that the transpose top with random shuffle is irreducible and aperiodic. Thus, $\mathscr{P}^{*k}$, the distribution after $k$ transitions converges to the unique stationary distribution $U_{S_n}$ as $k\rightarrow\infty$. Therefore, Theorem \ref{thm:coupling} ensures the existence of a successful maximal coupling $(\mathcal{X}_k,\mathcal{Y}_k)_k$ for the transpose top with random shuffle (on $S_n$) such that
	\begin{itemize}
		\item $(\mathcal{X}_k)_k$ starts at identity permutation and $(\mathcal{Y}_k)_k$ starts at $\zeta$.
		\item Both $\mathcal{X}_k$ and $\mathcal{Y}_k$ evolves according to the law $\mathscr{P}$, i.e., $\mathcal{X}_k\sim\mathscr{P}^{*k}$ and $\mathcal{Y}_k\sim U_{S_n}$. 
		\item $\mathcal{T}_{\text{couple}}$ is the coupling time, i.e., $\mathcal{X}_k=\mathcal{Y}_k\text{ for all }k\geq \mathcal{T}_{\text{couple}}$, $\mathbb{P}(\mathcal{T}_{\text{couple}}<\infty)=1$, and
		\begin{equation}\label{eq:TTR-coupling-ineq}
			||\mathscr{P}^{*k}-U_{S_n}||_{\text{TV}}=\mathbb{P}(\mathcal{T}_{\text{couple}}>k).
		\end{equation}
	\end{itemize}
	Define $L_{k+1}$ by $\mathcal{X}_{k+1}=\mathcal{X}_k\cdot(L_{k+1},n)$, and $R_{k+1}$ by
	$\mathcal{Y}_{k+1}=\mathcal{Y}_k\cdot(R_{k+1},n)$. In other words, $L_{k+1}$ is obtained from $\mathcal{X}_k^{-1}\mathcal{X}_{k+1}$, and $R_{k+1}$ is obtained from $\mathcal{Y}_k^{-1}\mathcal{Y}_{k+1}$. Thus, the sequence $\{L_k\}_{k=1}^{\infty}$ (also, $\{R_k\}_{k=1}^{\infty}$) consists of independent uniformly distributed random variables taking values from $\{1,\dots,n\}$. We note that, $L_{k}$ and $R_{k}$ depend on each other via the coupling $(\mathcal{X}_k,\mathcal{Y}_k)_k$.
	
	The known upper bound of the transpose top with random shuffle on $S_n$ (see \cite[Theorem 5.1]{D1}) provides
	\[\left|\left|\mathscr{P}^{*\lceil n\log n+Cn\rceil}-U_{S_n}\right|\right|_{\text{TV}}\leq ae^{-C}\quad\text{ for a universal constant }a\text{ and }C>1.\]
	Therefore, using \eqref{eq:TTR-coupling-ineq}, we have the following:
	\begin{equation}\label{eq:TTR-coupling-UB}
		\mathbb{P}(\mathcal{T}_{\text{couple}}>\lceil n\log n+Cn\rceil)\leq ae^{-C}.
	\end{equation}
	 
	Now we construct the coupling $(\mathscr{X}_k,\mathscr{Y}_k)_k$ as follows:
	\begin{itemize}
		\item Recall that $\mathscr{X}_0=(\G1,\dots,\G1;\1)$, and $(\mathscr{Y}_k)_k$ starts from a random element of $\Gn$.
		\item For $k\geq 0$, set
		\begin{equation*}
			\mathscr{X}_{k+1}=
			\begin{cases}
					\mathscr{X}_k\cdot (\G1,\underset{L_{k+1}\text{th position}}{\underset{\uparrow}{\dots\G1,\mathfrak{g}_{k+1},\G1,\dots}}\mathfrak{g}^{-1}_{k+1};(L_{k+1},n))&\quad\text{if }L_{k+1}<n,\\
					&\\
					\mathscr{X}_k\cdot (\G1,\dots,\G1,\dots,\mathfrak{g}_{k+1};\1)&\quad\text{if }L_{k+1}=n,
			\end{cases}
		\end{equation*}
	\begin{equation*}
		\mathscr{Y}_{k+1}=
		\begin{cases}
			\mathscr{Y}_k\cdot (\G1,\underset{R_{k+1}\text{th position}}{\underset{\uparrow}{\dots\G1,\mathfrak{h}_{k+1},\G1,\dots}}\mathfrak{h}^{-1}_{k+1};(R_{k+1},n))&\quad\text{if }R_{k+1}<n,\\
			&\\
			\mathscr{Y}_k\cdot (\G1,\dots,\G1,\dots,\mathfrak{h}_{k+1};\1)&\quad\text{if }R_{k+1}=n,
		\end{cases}
	\end{equation*}
	where we choose $\mathfrak{g}_{k+1}$ such that $L_{k+1}$th position $\mathscr{X}^{-1}_{k+1}$ and $\mathscr{Y}^{-1}_{k+1}$ matches; also choose $\mathfrak{h}_{k+1}$ such that $R_{k+1}$th position $\mathscr{X}^{-1}_{k+1}$ and $\mathscr{Y}^{-1}_{k+1}$ matches. We note that $L_{k+1}=R_{k+1}$ may happen.
	\end{itemize}
	Under the above coupling $(\mathscr{X}_k,\mathscr{Y}_k)_k$, if a position $i\in\{1,\dots,n-1\}$ matches in $\mathscr{X}^{-1}_{k_0}$ and $\mathscr{Y}^{-1}_{k_0}$, then the position $i$ agrees in $\mathscr{X}^{-1}_{k}$ and $\mathscr{Y}^{-1}_{k}$ for $k\geq k_0$. Therefore, we have the following
	\begin{itemize}
		\item First $n$ positions of $\mathscr{X}^{-1}_{k}$ and $\mathscr{Y}^{-1}_{k}$ will be matched when $k\geq \mathscr{T}'$ (defined in \eqref{eq:twisted-coupon-collector-time}), as they are updated by time $\mathscr{T}'$ with the final update at position $n$.
		\item The last position of $\mathscr{X}_k$ and $\mathscr{Y}_k$ matches when  $k=\mathcal{T}_{\text{couple}}$ (the coupling time for $(\mathcal{X}_k,\mathcal{Y}_k)_k$), but does not match when $k<\mathcal{T}_{\text{couple}}$.
	\end{itemize}  
	Thus, the coupling time $\mathfrak{T}_{\text{couple}}$ for $(\mathscr{X}_k,\mathscr{Y}_k)_k$ can not exceed $\max\{\mathcal{T}_{\text{couple}},\mathscr{T}'\}$,
	\[\text{i.e., }\mathscr{X}_k=\mathscr{Y}_k\text{ for all }k\geq\max\{\mathcal{T}_{\text{couple}},\mathscr{T}'\}.\]
	Hence by using \eqref{eq:coupling-ineq}, we have
	\begin{align}\label{eq:TV-UB1}
		\left|\left|P^{*\left\lceil n\log n+Cn\right\rceil}-U_{\Gn}\right|\right|_{\text{TV}}&\leq\mathbb{P}\left(\mathfrak{T}_{\text{couple}}>\left\lceil n\log n+Cn\right\rceil\right)\nonumber\\
		&\leq\mathbb{P}\left(\max\{\mathcal{T}_{\text{couple}},\mathscr{T}'\}>\left\lceil n\log n+Cn\right\rceil\right)\nonumber\\
		&\leq\mathbb{P}\left(\mathcal{T}_{\text{couple}}>\left\lceil n\log n+Cn\right\rceil\right)+\mathbb{P}\left(\mathscr{T}'>\left\lceil n\log n+Cn\right\rceil\right)\nonumber\\
		&\leq ae^{-C}+(e+1)e^{-\frac{C}{2}},
	\end{align}
	for $C>1$ and the constant $a$ in \eqref{eq:TTR-coupling-UB}.
	The inequality in \eqref{eq:TV-UB1} follows from \eqref{eq:TTR-coupling-UB} and Lemma \ref{lem:twisted-coupon-collector-time}. The first part of this theorem follows from the fact that $\left|\left|P^{*k}-U_{\Gn}\right|\right|_{\text{TV}}$ decreases as $k$ increases, and the second part follows from the first part by setting $C=\epsilon\log n-\frac{1}{n}$.
\end{proof}
\begin{rem}
	We now describe an explicit coupling (not necessarily optimal) for the transpose top with random shuffle on $S_n$, that provides a proof of \eqref{eq:TTR-coupling-UB} without using Theorem \ref{thm:coupling}. We use the same notation $(\mathcal{X}_k,\mathcal{Y}_k)_k$ (respectively $\mathcal{T}_{\text{couple}}$) for the coupling (respectively coupling time). Let $\mathcal{X}_0=\1$ and $\mathcal{Y}_0=\zeta$. The coupling description is as follows: Let $\mathcal{D}_k:=\{j:\mathcal{X}_k(j)=\mathcal{Y}_k(j)\}$ i.e., the set of positions where $\mathcal{X}_k$ and $\mathcal{Y}_k$ agree. Now choose $i\in\{1,\dots,n\}$ uniformly at random.
	\begin{itemize}
		\item If $i,n\in\mathcal{D}_k$, then $\mathcal{X}_{k+1}=\mathcal{X}_k\cdot\1$ and $\mathcal{Y}_{k+1}=\mathcal{Y}_k\cdot\1$. Thus $\mathcal{D}_{k+1}=\mathcal{D}_k$.
		\item If $i\in\mathcal{D}_k$ but $n\notin\mathcal{D}_k$, then $\mathcal{X}_{k+1}=\mathcal{X}_k\cdot\1$ and $\mathcal{Y}_{k+1}=\mathcal{Y}_k\cdot(j_0,n)$, where $j_0:=\mathcal{Y}_k^{-1}\mathcal{X}_k(n)$. Thus $\mathcal{D}_{k+1}=\mathcal{D}_k\cup\{n\}$.
		\item If $i\notin\mathcal{D}_k$ but $n\in\mathcal{D}_k$, then $\mathcal{X}_{k+1}=\mathcal{X}_k\cdot(i,n)$ and $\mathcal{Y}_{k+1}=\mathcal{Y}_k\cdot(i,n)$. Thus $\mathcal{D}_{k+1}=\mathcal{D}_k\cup\{i\}\setminus\{n\}$.
		\item If $i\notin\mathcal{D}_k$ and $n\notin\mathcal{D}_k$, then $\mathcal{X}_{k+1}=\mathcal{X}_k\cdot\1$ and $\mathcal{Y}_{k+1}=\mathcal{Y}_k\cdot(j_0,n)$, where $j_0:=\mathcal{Y}_k^{-1}\mathcal{X}_k(n)$. Thus $\mathcal{D}_{k+1}=\mathcal{D}_k\cup\{n\}$.
	\end{itemize}
	Thus we have the following: Each element of $\{1,\dots,n\}$ are equally probable to be added to $\mathcal{D}_{k+1}$ if $n\in\mathcal{D}_k$, $j(\neq n)\in\mathcal{D}_{k_0}$ implies $j\in\mathcal{D}_k$ for all $k\geq k_0$, $n\notin\mathcal{D}_k$ implies $n\in\mathcal{D}_{k+1}$, and $\mathcal{X}_k=\mathcal{Y}_k$ when $\mathcal{D}_k=\{1,\dots,n\}$. Therefore $\mathcal{T}_{\text{couple}}\leq\mathscr{T}+n$, where $\mathscr{T}$ is the coupon collector random variable satisfying \eqref{eq:coupon-collector}. Finally using the fact \[\mathbb{P}\left(\mathcal{T}_{\text{couple}}>\lceil n\log n+Cn\rceil\right)\leq \mathbb{P}\left(\mathscr{T}+n>\lceil n\log n+Cn\rceil\right),\]
	$\lceil n\log n+Cn\rceil-n=\lceil n\log n+(C-1)n\rceil$, and \eqref{eq:coupon-collector}; we can conclude \eqref{eq:TTR-coupling-UB} with $a=e$.
\end{rem}
\section{Lower bound for total variation distance}\label{sec:lower bound}
This section will focus on the lower bound of $\left|\left|P^{*k}-U_{\Gn}\right|\right|_{\text{TV}}$ for $k=n\log n+cn,\;c\ll0$ and prove Theorem \ref{thm:cutoff_main}. The idea is to use the fact that a projected chain mixes faster than the original chain. We define a group homomorphism from $\Gn$ onto the symmetric group $S_n$ which projects the warp-transpose top with random shuffle on $\Gn$ to the transpose top with random shuffle on $S_n$. Recall the lower bound results for the transpose top with random shuffle on $S_n$ from \cite[Section 4]{FTTR}. Although the detailed analysis for the transpose top with random shuffle on $S_n$ has appeared first in \cite[Chapter 5(C), p. 27]{D1}, the lower bound results we need here are directly available in \cite[Section 4]{FTTR}.

Recall that the transpose top with random shuffle is the random walk on $S_n$ driven by $\mathscr{P}$ (defined in \eqref{eq:defn_of_script_P_a}), and $\mathscr{P}^{*k}$ is the distribution after $k$ transitions. Then  $\mathscr{P}^{*k}\rightarrow U_{S_n}$ as $k\rightarrow\infty$.
Given $\pi\in S_n$, if $\mathfrak{f}(\pi)$ denotes the number of fixed points in $\pi$, then we have
\begin{equation}\label{eq:S_n-LB1}
	||\mathscr{P}^{*k}-U_{S_n}||_{\text{TV}}\geq 1-\frac{4\left(E_{k}\left(\mathfrak{f}^2\right)-\left(E_{k}\left(\mathfrak{f}\right)\right)^2\right)}{\left(E_{k}(\mathfrak{f})\right)^2}-\frac{2}{E_{k}(\mathfrak{f})}\quad\text{(see \cite[Proposition 4.1]{FTTR})},
\end{equation}
where $E_{k}(\mathfrak{f})$ (respectively $E_{k}(\mathfrak{f}^2)$) denotes the expected value of $\mathfrak{f}$ (respectively $\mathfrak{f}^2$) with respect to the probability measure $\mathscr{P}^{*k}$. Now recall the expressions for $E_{k}(\mathfrak{f})$ and $E_{k}(\mathfrak{f}^2)$ obtained in \cite{FTTR}.
\begin{align}
	E_{k}(\mathfrak{f})&\approx 1+(n-2)e^{-\frac{k}{n}}.\label{eq:E_k}\\
	E_{k}(\mathfrak{f}^2)&\approx  2+3(n-2)e^{-\frac{k}{n}}+(n^2-5n+5)e^{-\frac{2k}{n}}+(n-2)\left(\frac{1+(-1)^k}{n^k}\right).\label{eq:V_k}
\end{align}
Let us define a homomorphism $f$ from $\Gn$ onto $S_n$ as follows:
\begin{equation}\label{eq:definition_of_projection}
	f: (g_1,\dots,g_n;\pi)\mapsto \pi,\text{ for }(g_1,\dots,g_n;\pi)\in \Gn.
\end{equation}
It can be checked that the mapping $f$ defined in \eqref{eq:definition_of_projection} is a surjective homomorphism. Moreover, $f$ projects the warp-transpose top with random shuffle on $\Gn$ to the transpose top with random shuffle on $S_n$ i.e., $Pf^{-1}=\mathscr{P}$. Here $f^{-1}$ is defined by $f^{-1}(\pi):=\{\pi'\in\Gn:f(\pi')=\pi\}$ for $\pi\in S_n$. We now prove a lemma which will be useful in proving the main result of this section.
\begin{lem}\label{lem:transition_preservation_of_the_projection}
	For any positive integer $k$ we have $\left(Pf^{-1}\right)^{*k}=P^{*k}f^{-1}$.
\end{lem}
\begin{proof}
	We use the first principle of mathematical induction on $k$. The base case for $k=1$ is true by definition. Now assume the induction hypothesis i.e., $\left(Pf^{-1}\right)^{*m}=P^{*m}f^{-1}$ for some positive integer $m>1$. Let $\pi\in S_n$ be chosen arbitrarily. Then for the inductive step $k=m+1$ we have the following:
	\begin{align}\label{eq:transition_preservation-1}
		\left(Pf^{-1}\right)^{*(m+1)}(\pi)&=\left(\left(Pf^{-1}\right)*\left(Pf^{-1}\right)^{*m}\right)(\pi)\nonumber\\
		&=\sum_{\{\xi,\zeta\in S_n:\;\xi\zeta=\pi\}}\left(Pf^{-1}\right)(\xi)\left(Pf^{-1}\right)^{*m}(\zeta)\nonumber\\
		&=\sum_{\{\xi,\zeta\in S_n:\;\xi\zeta=\pi\}}\left(Pf^{-1}\right)(\xi)\left(P^{*m}f^{-1}\right)(\zeta),\;\text{ by the induction hypothesis},\nonumber\\
		&=\sum_{\substack{\xi,\zeta\in S_n\\\xi\zeta=\pi}}P\left(f^{-1}(\xi)\right)P^{*m}\left(f^{-1}(\zeta)\right)\nonumber\\
		&=\sum_{\substack{\xi,\zeta\in S_n\\\xi\zeta=\pi}}\sum_{\substack{\xi^{\prime}\in f^{-1}(\xi)\\\zeta^{\prime}\in f^{-1}(\zeta)}}P(\xi^{\prime})P^{*m}(\zeta^{\prime}).
	\end{align}
	Now using the fact that $f$ is a homomorphism, we have the following:
	\begin{align*}
		\{(\xi^{\prime},\zeta^{\prime})\in f^{-1}(\xi)\times f^{-1}(\zeta):\;\xi,\zeta\in S_n\text{ and }\xi\zeta=\pi\}=\{(\xi^{\prime},\zeta^{\prime})\in \Gn\times \Gn:\;\xi^{\prime}\zeta^{\prime}\in f^{-1}(\pi)\}.
	\end{align*}
	Therefore the expression in \eqref{eq:transition_preservation-1} becomes
	\begin{align*}
		\sum_{\{\xi^{\prime},\zeta^{\prime}\in \Gn:\;\xi^{\prime}\zeta^{\prime}\in f^{-1}(\pi)\}}P(\xi^{\prime})P^{*m}(\zeta^{\prime})&=\sum_{\pi^{\prime}\in f^{-1}(\pi)}\sum_{\substack{\xi^{\prime},\zeta^{\prime}\in \Gn\\\xi^{\prime}\zeta^{\prime}=\pi^{\prime}}}P(\xi^{\prime})P^{*m}(\zeta^{\prime})\\
		&=\sum_{\pi^{\prime}\in f^{-1}(\pi)}P^{*(m+1)}(\pi^{\prime})=\left(P^{*(m+1)}f^{-1}\right)(\pi).
	\end{align*}
	Thus the lemma follows from the first principle of mathematical induction.
\end{proof}
\begin{thm}\label{lem:G_n Lower Bound}
	For the random walk on $\Gn$ driven by $P$ we have the following:
	\begin{enumerate}
		\item For large $n$,
		\[||P^{*k}-U_{\Gn}||_{\emph{TV}}\geq1-\frac{2\left(3+3e^{-c}+o(1)(e^{-2c}+e^{-c}+1)\right)}{\left(1+(1+o(1))e^{-c}\right)^2},\]
		when $k=n\log n+cn$ and $c\ll0$.
		\item For any $\epsilon\in (0,1)$,
		\[\displaystyle\lim_{n\rightarrow\infty}\left|\left|P^{*\lfloor(1-\epsilon)n\log n\rfloor}-U_{\Gn}\right|\right|_{\emph{TV}}=1.\]
	\end{enumerate}
\end{thm}
\begin{proof}
	We know that, \emph{given two probability distributions $\mu$ and $\nu$ on $\Omega$ and a mapping $\psi:\Omega\rightarrow\Lambda$, we have $||\mu-\nu||_{\emph{TV}}\geq ||\mu\psi^{-1}-\nu\psi^{-1}||_{\emph{TV}}$, where $\Lambda$ is finite} \cite[Lemma 7.9]{LPW}. Therefore we have the following:
	\begin{align}\label{eq:LB-1}
		||P^{*k}-U_{\Gn}||_\text{{TV}}&\geq||P^{*k}f^{-1}-U_{\Gn}f^{-1}||_\text{{TV}}\nonumber\\
		&=||\left(Pf^{-1}\right)^{*k}-U_{S_n}||_\text{{TV}},\;\text{ by Lemma \eqref{lem:transition_preservation_of_the_projection} and }U_{\Gn}f^{-1}=U_{S_n},\nonumber\\
		&=||\mathscr{P}^{*k}-U_{S_n}||_\text{{TV}},
	\end{align}
	using $Pf^{-1}=\mathscr{P}$. Therefore \eqref{eq:S_n-LB1}, \eqref{eq:E_k}, \eqref{eq:V_k}, and \eqref{eq:LB-1} implies that
	\begin{align}\label{eq:LB-main}
		||P^{*k}-U_{\Gn}||_\text{{TV}}\geq&\;1-\frac{2\left(2E_{k}\left(\mathfrak{f}^2\right)-2\left(E_{k}\left(\mathfrak{f}\right)\right)^2+E_{k}\left(\mathfrak{f}\right)\right)}{\left(E_{k}(\mathfrak{f})\right)^2}\nonumber\\
		 =&\;1-\frac{2\left(3+3(n-2)e^{-\frac{k}{n}}-2(n-1)e^{-\frac{2k}{n}}+o(1)\right)}{\left(1+(n-2)e^{-\frac{k}{n}}\right)^2},\;\text{ for }k>1,
	\end{align}
	for sufficiently large $n$. The equality in \eqref{eq:LB-main} holds because of the following:
	\begin{align*}
		&2\left(E_{k}\left(\mathfrak{f}^2\right)-\left(E_{k}\left(\mathfrak{f}\right)\right)^2\right)+E_{k}\left(\mathfrak{f}\right)\\
		\approx&\;2\left(2+3(n-2)e^{-\frac{k}{n}}+(n^2-5n+5)e^{-\frac{2k}{n}}+(n-2)\left(\frac{1+(-1)^k}{n^k}\right)-\left(1+(n-2)e^{-\frac{k}{n}}\right)^2\right)\\
		&\hspace*{4.8in}+\left(1+(n-2)e^{-\frac{k}{n}}\right)\\
		=&\;3+3(n-2)e^{-\frac{k}{n}}-2(n-1)e^{-\frac{2k}{n}}+2(n-2)\left(\frac{1+(-1)^k}{n^k}\right).
\end{align*}
Now if $n$ is large, $c\ll0$ and $k=n\log n+cn$, then by \eqref{eq:LB-main}, we have the first part of this theorem.
	
	Again for any $\epsilon\in (0,1)$, from \eqref{eq:LB-main}, we have
	\begin{equation}\label{eq:LB cutoff limit}
		1\geq ||P^{*\lfloor(1-\epsilon)n\log n\rfloor}-U_{\Gn}||_{\text{TV}}\geq  1-\frac{2\left(3+3n^{\epsilon}+o(1)(n^{2\epsilon}+n^{\epsilon}+1)\right)}{\left(1+(1+o(1))n^{\epsilon}\right)^2},
	\end{equation}
	for large $n$. Therefore, the second part of this theorem follows from \eqref{eq:LB cutoff limit} and the fact that
	\[\lim_{n\rightarrow\infty}\frac{2\left(3+3n^{\epsilon}+o(1)(n^{2\epsilon}+n^{\epsilon}+1)\right)}{\left(1+(1+o(1))n^{\epsilon}\right)^2}=\lim_{n\rightarrow\infty}\frac{2\left(\frac{3}{n^{2\epsilon}}+\frac{3}{n^{\epsilon}}+o(1)(\frac{1}{n^{2\epsilon}}+\frac{1}{n^{\epsilon}}+1)\right)}{\left(\frac{1}{n^{\epsilon}}+1+o(1)\right)^2}=\;0.\qedhere\]
\end{proof}
\begin{proof}[Proof of Theorem \ref{thm:cutoff_main}]
	Theorem \ref{thm:cutoff_main} follows from the second part of Theorem \ref{thm:TV-main_UB} and the second part of Theorem \ref{lem:G_n Lower Bound}.
\end{proof}
\subsection*{Acknowledgement} I extend sincere thanks to my PhD advisor Arvind Ayyer for all the insightful discussions during the preparation of this paper. I am very grateful to the anonymous referees of the \emph{Journal of Algebraic Combinatorics} for many constructive suggestions. I would like to thank an anonymous referee of the \emph{Algebraic Combinatorics} for the valuable comments, which helped improve the total variation upper bound result and simplify the proof of the total variation lower bound. I am grateful to Professor Tullio Ceccherini-Silberstein for his encouragement and inspiring comments. I would also like to thank Guy Blachar, Ashish Mishra, and Shangjie Yang for their discussions. I would like to acknowledge support in part by a UGC Centre for Advanced Study grant.
\subsection*{Statements and Declarations} The author has no conflict of interest to disclose. This article is a part of author’s PhD dissertation. The extended abstract of this article was accepted in FPSAC 2020 (online).
\subsection*{Data Availability Statements}  Data sharing not applicable to this article as no datasets were generated or analysed during the current study.
\subsection*{Copyright Statements}  This version of the article has been accepted for publication in the Journal of Algebraic Combinatorics: An International Journal, after peer review but is not the Version of Record and does not reflect post-acceptance improvements, or any corrections. The Version of Record is available online at: \url{http://dx.doi.org/10.1007/s10801-023-01271-1}. Use of this Accepted Version is subject to the publisher’s Accepted Manuscript terms of use \url{https://www.springernature.com/gp/open-research/policies/accepted-manuscript-terms}.
\bibliography{ref}{}
\bibliographystyle{plain}
\end{document}